\documentclass[11pt,a4paper,oneside]{article}
\usepackage[fleqn, leqno]{amsmath}
\usepackage{bm}
\usepackage{amssymb}
\usepackage{amsthm}
\usepackage{eucal}
\usepackage{makeidx}
\usepackage{hyperref}
\usepackage{xcolor}
\makeindex

\usepackage{fancyhdr}
\fancypagestyle{plain}{%
	\fancyhead{} 
}

\theoremstyle{plain}
\newtheorem{theorem}{Theorem}[section]
\newtheorem{proposition}[theorem]{Proposition}

\newtheorem{conjecture}[theorem]{Conjecture}

\theoremstyle{definition}
\newtheorem{definition}[theorem]{Definition}

\theoremstyle{remark}

\newtheorem{example}[theorem]{Example}

\newcommand{\N}{\mathbb{N}}
\newcommand{\Z}{\mathbb{Z}}

\newcommand{\bi}{\mathfrak{i}}

\begin{document}


\title{Dynamic Structures of Monomials \\ on $p$-adic Integers for Small Primes $p$}

\author{Myunghyun Jung and Donggyun Kim}

\date{}
\renewcommand\thefootnote{}


\maketitle \setlength{\hangindent}{0pt}

\begin{abstract}
We study the dynamic structures of the monomial $x^m$ over the ring of $p$-adic integers for every positive integer $m$ and for primes $p=2,3$ and $5$. The dynamic structures are described by investigating minimal decompositions which consist of minimal subsystems and attracting basins.
\end{abstract}



\tableofcontents




\section{Introduction}

Let $\Z_p$ be the ring of $p$-adic integers and $f$ a polynomial over $\Z_p$. The dynamical system $(\Z_p, f)$ has a general theorem of the minimal decomposition due to Fan and Liao in 2011 \cite{FanLiao2011Minimal}.

\begin{theorem} [\cite{FanLiao2011Minimal}] \label{thm:decomposition}
Let $f \in \mathbb{Z}_p[x]$ be a polynomial of integral coefficients with degree at least  $2$. We have the following decomposition
\[
\mathbb{Z}_p = \mathcal{P} \bigsqcup \mathcal{M} \bigsqcup \mathcal{B}
\]
where
\begin{enumerate}
\item $ \mathcal{P}$ is the finite set consisting of all periodic points of $f$,
\item $\mathcal{M}= \bigsqcup_i \mathcal{M}_i$ is the union of all (at most countably many) clopen invariant sets such that each $\mathcal{M}_i$ is a finite union of balls and each subsystem $f: \mathcal{M}_i \to \mathcal{M}_i$ is minimal, and
\item each point in $\mathcal{B}$ lies in the attracting basin of a periodic orbit or of a minimal subsystem.
\end{enumerate}
\end{theorem}
This decomposition is usually referred to as a \emph{minimal decomposition} and the invariant subsets $\mathcal{M}_i$ are called {\em minimal components }.
	
Although we have a general decomposition theorem, it is still difficult to describe  the explicit minimal decomposition for  a given polynomial,  as there have been few works done on them. Multiplications on $\Z_p$ for $p \geq 3$ were studied by Coelho and Parry in 2001 \cite{CoelhoParry2001}, and affine maps and quotient maps of affine maps for all $p$ were studied by Fan, Li, Yao and Zhou in 2007 \cite{FanLiYaoZhou2007Affine} and by Fan, Fan, Liao and Wang in 2014 \cite{FanFanLiaoWang2014Homographic}, respectively. The quadratic maps on $\Z_p$ for $p=2$ were investigated by Fan and Liao in 2011 \cite{FanLiao2011Minimal}, respectively. Recently, the Chebyshev polynomials and Fibonacci polynomials on $\Z_2$ were investigated by Fan and Liao in 2016 \cite{FanLiao2016Chebyshev} and by Jung, Kim and Song in 2019 \cite{JungKimSong2019Fibonacci}, respectively. The monomial $x^m$ on $\Z_p$ is a basic function, but it's dynamic structures are not studied except only the case $x^2$ for every $p$ by Fan and Liao in 2016 \cite{FanLiao2016Square}.

In this paper, we study the dynamic structures of the monomial $x^m$ over the ring of $p$-adic integers for every positive integer $m$ and for primes $p=2,3$ and $5$. The dynamic structures are described by investigating minimal decompositions which consist of minimal subsystems and attracting basins following the style in Theorem \ref{thm:decomposition}.

\section{Preliminary}
\subsection{$p$-adic spaces}
We first review the basic definitions and properties of the $p$-adic space.
Let $\N$ be the set of positive integers, $\Z$ the set of integers and $p$ a prime number. We consider a family of rings $\{ \Z/p^k \Z\}$ for all $k\in \N$ with canonical projections $\{ \varphi_{i j} : \Z/p^j \Z \to \Z/p^i \Z \}$ for all integers $i$ and $j$ with $i\le j$. The \emph{ring of $p$-adic integers}, denoted $\Z_p$, is the projective limit of the projective system of rings $\{\Z/p^k \Z\}$ with projections $\{\varphi_{i j} \}$.

Then the ring $\Z_p$ is the set of all sequence $(x_k)$ with the property $x_k \in \Z/p^k \Z$ and $\varphi_{i j}(x_j)=x_i$ for all $i \le j$. A point $x_k$ in $\Z/p^k \Z$ can be represented by $x_k=a_0 + a_1 p+a_2 p^2 + a_3 p^3 + \cdots+a_{k-1}p^{k-1}$ with unique $a_i \in\{0,1,\cdots,p-1\}$ and the point $x = (x_k)$ in $\Z_{p}$ by $x: a_0 , a_0 + a_1 p, a_0 + a_1 p+a_2 p^2, a_0 + a_1 p+a_2 p^2 + a_3 p^3, \cdots $ as a sequence form, or simply $x=a_0 + a_1 p+a_2 p^2 + a_3 p^3 + \cdots$.

An element $x$ can be written as $x=p^\nu (b_0 + b_1 p+b_2 p^2 + b_3 p^3 + \cdots)$ with $\nu\in \N\cup \{0\}$ and $b_i \in\{0,1,\cdots,p-1\}$. In this case, we say that $p^\nu$ \emph{divides} $x$ and denote $p^\nu |x$.

The \emph{$p$-adic order} or \emph{$p$-adic valuation} for $\Z_p$ is defined as the function $\nu_p : \Z_p \rightarrow \mathbb{Z} \cup \infty$
\[
\nu_p (x) = \begin{cases}\textrm{max}\{\nu \in \N\cup \{0\} : p^\nu | x \}&\textrm{if}~x \neq 0,\\
\infty &\textrm{if}~x =0,\end{cases}
\]

and the \emph{$p$-adic metric} of $\Z_p$ is defined as the function $|\cdot |_p : \Z_p \rightarrow \mathbb{R}$
\[
|x|_p = \begin{cases}p^{-\nu_p (x)} &\textrm{if}~x \neq 0,\\
0 &\textrm{if}~x =0.\end{cases}
\]
With the $p$-adic metric, $\Z_p$ becomes a compact space.

Note that $\Z/p^k\Z$ can be viewed as a subset of $\Z_{p}$,
\[ \Z/p^k\Z=\{x\in\Z_{p}: a_{i}=0 \text{ for every } i\ge k\}, \]
and for $x=a_0 + a_1 p+a_2 p^2 + a_3 p^3 + \cdots$ in $\Z_p$, the element  $x \pmod{ p^k}$ means $a_0 + a_1 p+a_2 p^2 + a_3 p^3 + \cdots+a_{k-1}p^{k-1}$, which is now in $\Z_p / p^k \Z_p$.
It can be checked that $\Z/p^k\Z$ is isomorphic to $\Z_p / p^k \Z_p$. We state properties of the $p$-adic valuation. 

\begin{proposition}
Let $x$ and $y$ be elements of $\Z_p$. Then, $\nu_p$ has the properties:
\begin{enumerate}
  \item $\nu_p(xy)=\nu_p(x)+\nu_p(y).$
  \item $\nu_p(x+y)\geq \min\{\nu_p(x),\nu_p(y)\}$, with equality when $\nu_p(x)$ and $\nu_p(y)$ are unequal.
\end{enumerate}
\end{proposition}


\subsection{Dynamics of $p$-adic polynomials}

Let $f$ be a polynomial over $\Z_p$ and let $l$ be a positive integer. Let $\sigma=\{x_1,\cdots,x_k\}$ with $x_i \in \Z/p^l\Z$ be a cycle of $f$ of length $k$ (also called a $k$-cycle), i.e.,
$$f(x_1)=x_2,\cdots,f(x_i)=x_{i+1},\cdots,f(x_k)=x_1~(\mathrm{mod}~p^l).$$
In this case, we also say that the cycle $\sigma$ is \emph{at level l}. Let
\[ X_\sigma:=\bigcup_{i=1}^k X_i~\mathrm{where}~X_i:=\{x_i+p^lt+p^{l+1}\Z:~t=0,\cdots,p-1\}. \]
Then
\[ f(X_i)\subset X_{i+1}~(1\leq i\leq k-1)~\mathrm{and}~f(X_k)\subset X_1, \]
and so $f$ is invariant on the clopen set $X_\sigma$. The cycles in $X_\sigma$ of $f$ (mod $p^{l+1}$) are called \emph{lifts of $\sigma$} (from level $l$ to level $l+1$). Note that the length of a lift of $\sigma$ is a multiple of $k$.

Let $\mathbb{X}_i:=x_i+p^l\Z_p$ and let $\mathbb{X}_{\sigma}:=\bigcup_{i=1}^{k}\mathbb{X}_i$. For $x\in\mathbb{X}_{\sigma}$, denote
\begin{align}
&a_l(x):=\prod_{j=0}^{k-1}f'\big(f^j(x)\big) \label{a_l}
\mathrm{~and~}\\
&b_l(x):=\frac{f^k(x)-x}{p^l} \label{b_l}.
\end{align}
Sometimes we abbreviate $a_l(x)$ and $b_l(x)$ to $a_l$ and $b_l$, respectively. The following definition and proposition were treated by Fan et al \cite{FanFares2011, FanLiYaoZhou2007Affine}. The following definition is for $p=2$.

\begin{definition}\label{def:movement for p=2}

We say $\sigma$ \emph{strongly grows at level l} if $a_l\equiv 1~(\mathrm{mod}~4)$ and $b_l\equiv 1~(\mathrm{mod}~2)$.
	
We say $\sigma$ \emph{strongly splits at level l} if $a_l\equiv 1~(\mathrm{mod}~4)$ and $b_l\equiv 0~(\mathrm{mod}~2)$.
	
We say $\sigma$ \emph{weakly grows at level l} if $a_l\equiv 3~(\mathrm{mod}~4)$ and $b_l\equiv 1~(\mathrm{mod}~2)$.
	
We say $\sigma$ \emph{weakly splits at level l} if $a_l\equiv 3~(\mathrm{mod}~4)$ and $b_l\equiv 0~(\mathrm{mod}~2)$.
	
We say $\sigma$ \emph{grows tails at level l} if $a_l\equiv 0~(\mathrm{mod}~2)$.
\end{definition}

\begin{proposition}[\cite{FanLiao2011Minimal},\cite{FanLiao2016Chebyshev}]\label{prop:behavior of lifts for p=2}
Let $\sigma=\{x_1,\cdots,x_k\}$ be a cycle of $f$ at level $l$.
\begin{enumerate}
\item If $\sigma$ is a strongly growing cycle at level $l\geq2$, then $f$ restricted onto the invariant clopen set $\mathbb{X}_{\sigma}$ is minimal.
		
\item If $\sigma$ is a strongly splitting cycle at level $l\geq2$ and $\nu_2(b_l(x))=s\geq1$ for all $x\in\mathbb{X}_\sigma$, then the lifts of $\sigma$ will keep splitting until to the level $l+s$ at which all lifts strongly grow.
		
\item Let $\sigma$ be a cycle of $f$ at level $n\geq2$. If $\sigma$ weakly grows then the lift of $\sigma$ strongly splits.
		
\item Let $\sigma$ be a weakly splitting cycle of $f$ at level $n\geq2$. Then one lift behaves the same as $\sigma$ while the other weakly grows and then strongly splits.
		
\item If $\sigma$ is a growing tails cycle at level $l\geq1$, then $f$ has a $k$-periodic orbit in the clopen set $\mathbb{X}_{\sigma}$ lying in its attracting basin.
\end{enumerate}
\end{proposition}

The following definition is for $p\geq 3$.

\begin{definition}[\cite{Desjardins2001polynomial}]\label{def:movement for p>=3}
We say $\sigma$ \emph{grows at level l} if $a_l\equiv 1~(\mathrm{mod}~p)$ and $b_l\not\equiv 0~(\mathrm{mod}~p)$.
	
We say $\sigma$ \emph{splits at level l} if $a_l\equiv 1~(\mathrm{mod}~p)$ and $b_l\equiv 0~(\mathrm{mod}~p)$.
	
We say $\sigma$ \emph{grows tails at level l} if $a_l\equiv 0~(\mathrm{mod}~p)$.
	
We say $\sigma$ \emph{partially splits at level l} if $a_l\not\equiv 0,1~(\mathrm{mod}~p)$.
\end{definition}

\begin{proposition}[\cite{Desjardins2001polynomial}]\label{prop:behavior of lifts for p>=3}
Let $\sigma=\{x_1,\cdots,x_k\}$ be a cycle of $f$ at level $l\geq 1$.
\begin{enumerate}
\item If $\sigma$ is a growing cycle at level $l\geq 2$, then $f$ restricted to $X_{\sigma}$ (mod $p^{l+1}$) consists of a single cycle of length $pk$.

\item If $\sigma$ is a splitting cycle at level $l\geq 1$, then $f$ restricted to $X_{\sigma}$ (mod $p^{l+1}$) consists of $p$ cycles, each of length $k$.

\item If $\sigma$ is a growing tails cycle at level $l\geq1$, then $f$ restricted to $X_{\sigma}$ (mod $p^{l+1}$) contains one $k$-cycle, and the remaining points of $X$ are mapped into this cycle by $f^k$.

\item If $\sigma$ is a partially splitting cycle at level $l\geq1$, then $f$ restricted to $X_{\sigma}$ (mod $p^{l+1}$) contains one $k$-cycle and $(p-1)/d$ cycles of length $kd$, where $d$ is the order of $a_l$ in $(\Z/p\Z)^*$.
\end{enumerate}
\end{proposition}

\begin{proposition}[\cite{Desjardins2001polynomial}]\label{prop: lifts of splitting cycle for p>=3}
Let $p\geq 3$ and $l\geq 1$. Let $\sigma$ be a splitting cycle of $f$ at level $l$. Define
\begin{equation}\label{def: A_l and B_l}
A_l(x)=\nu_p\big(a_l(x)-1\big),\quad B_l(x)=\nu_p\big(b_l(x)\big).
\end{equation}
\begin{enumerate}
\item If $ B_l < A_l$ and $B_l < l$, every lift splits $B_l-1$ times then all lifts at level $B_l+l$ grow forever.

\item If $A_l\leq B_l$ and $A_l<l$, there is one lift which behaves the same as $\sigma$ (i.e., this lift splits and $A_{l+1}\leq B_{l+1}$ and $A_{l+1}<l+1$) and other lifts split $A_l-1$ times then all lifts at level $A_l+l$ grow forever.

\item If $l\leq A_l$ and $l\leq B_l$, then all lifts split at least $l-1$ times.
\end{enumerate}
\end{proposition}

\section{Minimal decompositions of monomials on $\Z_2$}

From this section, we study the dynamical structure of monomials $f(x)=x^m$ on the ring of $p$-adic integers $\Z_p$ for $m \in \N$ through the minimal decomposition which consists of periodic points, minimal components and attracting basins, see Theorem \ref{thm:decomposition}.

When $m=1$, that is, $f(x)=x$, every point in $\Z_p$ is a fixed point, whence the minimal decomposition is trivial. Therefore from now on, we consider the cases $m\geq 2$.

For arbitrary prime number $p$, we have the following.

\begin{proposition}\label{prop:pZ_p attracting basin}
Let $m\geq 2$ be an integer and $f(x)=x^m$. Then, $f(x)$ has a fixed point 0 in the clopen set $p\Z_p$ with $p\Z_p$ lying in its attracting basin.
\end{proposition}

\begin{proof}
Since $f(0)=0$, $\{0\}$ is a cycle of length 1 at level 1.
Now we compute the quantity $a_1$ for the cycle $\{0\}$, as defined in (\ref{a_l}). Since $f'(x)=mx^{m-1}$, then $a_1(0)=0$. Therefore, $\{0\}$ grows tails at level 1 by Definition \ref{def:movement for p=2} for $p=2$ and Definition \ref{def:movement for p>=3} for $p\geq 3$. Then, by statement 3 of Proposition \ref{prop:behavior of lifts for p>=3}, the statement is proved.
\end{proof}

In this section, we consider the case $p=2$. For integers $m\geq 2$, we divide into cases that $m\equiv 0 \pmod{2},\ m\equiv 1 \pmod{4}$ and $m\equiv -1 \pmod{4}$.

\begin{theorem}\label{thm:Z_2 m=0(2)}
Let $f(x)=x^m$ over $\Z_2$ and assume $m\geq 2$ with $m\equiv 0~(\mathrm{mod}~2)$. Then, the minimal decomposition of $\Z_2$ for $f(x)$ is
\[ \Z_2=\{0,1\}\bigsqcup\Big((2\Z_2-\{0\})\cup\big((1+2\Z_2)-\{1\}\big)\Big). \]
Here, $\{0,1\}$ is the set of fixed points and the other sets are the attracting basin of $\{0,1\}$.
\end{theorem}

\begin{proof}
Since $m$ is even, we obtain that $f(1)=1$ and $a_1(1)=f'(1)\equiv 0~(\mathrm{mod}~2)$ by notation (\ref{a_l}). So, $\{1\}$ is a cycle of length 1 at level 1 which grows tails by Definition \ref{def:movement for p=2}. By property 5 of Proposition \ref{prop:behavior of lifts for p=2}, $f(x)$ has a fixed point 1 in the clopen set $1+2\Z_2$ with $1+2\Z_2$ lying in its attracting basin. With Proposition \ref{prop:pZ_p attracting basin} the proof is completed.
\end{proof}

Consider the case $m\equiv 1~(\mathrm{mod}~4)$.

\begin{proposition}\label{prop:Z_2 m=1(4)}
Let $f(x)=x^m$ over $\Z_2$ and assume $m\geq 2$ with $m\equiv 1~(\mathrm{mod}~4)$. Let $t=\nu_2(m-1)$. For every nonnegative integer $l$, there are strongly splitting 1-cycles $\{1+2^{l+2}\}$ and $\{-1+2^{l+2}\}$ at level $l+3$ such that all the lifts at level $t+l+2$ strongly grow.
\end{proposition}

\begin{proof}
Since $m\equiv 1~(\mathrm{mod}~4)$ we have $t \geq 2$. Write $m=2^t k+1$, then $k$ is odd and $f(x)=x^{2^t k+1}$. We compute that $f(\pm 1)=(\pm 1)^{2^t k+1}=\pm 1$. So, $\pm 1$ are fixed points at any level. Let $l$ be any nonnegative integer and view $\{1\}$ and $\{-1\}$ be cycles of length 1 at level $l+2$.

We compute the quantity $a_{l+2}$, as defined in (\ref{a_l}),
\begin{align*}
a_{l+2}(\pm 1)=f'(\pm 1)=(2^tk+1)(\pm 1)^{2^t k}\equiv 1~(\mathrm{mod}~4),
\end{align*}
and the quantity $b_{l+2}$, as defined in (\ref{b_l}),
\begin{align*}
b_{l+2}(\pm 1)=\frac{f(\pm 1)-(\pm 1)}{2^{l+2}}\equiv 0~(\mathrm{mod}~2).
\end{align*}
By Definition \ref{def:movement for p=2}, the cycles $\{1\}$ and $\{-1\}$ strongly split at level $l+2$.

The lifts of $\{1\}$ and $\{-1\}$ from level $l+2$ to level $l+3$ are $\{1, 1+2^{l+2}\}$ and $\{-1, -1+2^{l+2}\}$, respectively. The points $\pm 1+2^{l+2}$ are fixed points at level $l+3$.

Now we compute the quantity $a_{l+3}$ for the above points, as defined in (\ref{a_l}),
\begin{align*}
a_{l+3}(\pm 1+2^{l+2})&=(2^tk+1)(\pm 1+2^{l+2})^{2^tk}\\
&\equiv 1~(\mathrm{mod}~4).
\end{align*}
We compute the quantity $b_{l+3}$ for the above cycles, as defined in (\ref{b_l}),
\begin{align*}
f(\pm 1+2^{l+2})-(\pm 1+2^{l+2})&=(\pm 1+2^{l+2})^{2^tk+1}-(\pm 1+2^{l+2})\\
&=\sum_{j=0}^{2^tk+1}\tbinom{2^tk+1}{j}(\pm 1)^{2^tk+1-j}(2^{l+2})^j-(\pm 1+2^{l+2})\\
&\equiv 0~(\mathrm{mod}~2^{l+4}),
\end{align*}
So, $b_{l+3}(\pm 1+2^{l+2})=\frac{f(1+2^{l+2})-(1+2^{l+2})}{2^{l+3}}\equiv 0~(\mathrm{mod}~2)$. Hence, the 1-cycles $\{1+2^{l+2}\}$ and $\{-1+2^{l+2}\}$ strongly split at level $l+3$.

If we show the following expression, by statement 2 of Proposition \ref{prop:behavior of lifts for p=2}, the proof is complete:
\[ \nu_2\big(b_{l+3}(x)\big)=t-1\ \text{for all}\ x\in \pm 1+2^{l+2}+2^{l+3}\Z_2. \]
Suppose that $x=1+2^{l+2}+2^{l+3}a\in 1+2^{l+2}+2^{l+3}\Z_2$ with $a\in\Z_2$. Then
\begin{align*}
f(1+&2^{l+2}+2^{l+3}a)-(1+2^{l+2}+2^{l+3}a)\\
&=(1+2^{l+2}+2^{l+3}a)^{2^tk+1}-(1+2^{l+2}+2^{l+3}a)\\
&=\sum_{j=0}^{2^tk+1}\tbinom{2^tk+1}{j}(1+2^{l+2})^{2^tk+1-j}(2^{l+3}a)^j-(1+2^{l+2}+2^{l+3}a).
\end{align*}
For an integer $j\geq 2$, the summand has the valuation
\begin{align*}
\nu_2\big(\tbinom{2^tk+1}{j}&(1+2^{l+2})^{2^tk+1-j}(2^{l+3}a)^j\big)\\
&\geq t+\frac{j-1}{2}-1+j(l+3)-(j-1)\\
&\geq t+2l+4.
\end{align*}
So, we obtain
\begin{align*}
f(1+&2^{l+2}+2^{l+3}a)-(1+2^{l+2}+2^{l+3}a)\\
&\equiv(1+2^{l+2})^{2^tk+1}+\tbinom{2^tk+1}{1}(1+2^{l+2})^{2^tk}(2^{l+3}a)-(1+2^{l+2}+2^{l+3}a)\\
&\equiv\big((1+2^{l+2})^{2^tk}-1\big)(1+2^{l+2}+2^{l+3}a)\\
&\equiv\big(\sum_{j=0}^{2^tk}\tbinom{2^tk}{j}(2^{l+2})^j-1\big)(1+2^{l+2}+2^{l+3}a)~(\mathrm{mod}~2^{t+l+3}).
\end{align*}
For an integer $j\geq3$, the summand has the valuation
\begin{align*}
\nu_2\big(\tbinom{2^tk}{j}(2^{l+2})^j\big)&\geq t+\frac{j-1}{2}-1+j(l+2)-(j-1)\\
&\geq t+3l+4.
\end{align*}
So, we obtain
\begin{align*}
f(1+&2^{l+2}+2^{l+3}a)-(1+2^{l+2}+2^{l+3}a)\\
&\equiv \big(\tbinom{2^tk}{1}2^{l+2}+\tbinom{2^tk}{2}(2^{l+2})^2\big)(1+2^{l+2}+2^{l+3}a)\\
&\equiv 2^{t+l+2}~(\mathrm{mod}~2^{t+l+3}).
\end{align*}
Hence, $\nu_2\big(b_{l+3}(x)\big)=\nu_2\big(\frac{f(1+2^{l+2}+2^{l+3}a)-(1+2^{l+2}+2^{l+3}a)}{2^{l+3}}\big)=t-1$.

When $x=-1+2^{l+2}+2^{l+3}a\in -1+2^{l+2}+2^{l+3}\Z_2$ with $a\in\Z_2$, we can prove it similarly. Therefore the proof is completed.
\end{proof}

For the case $m\equiv -1 \pmod{4}$, we need the following.

\begin{proposition}[\cite{FanLiao2011Minimal}, Proposition 7]\label{ws wg and ss}
Let $\sigma$ be a weakly splitting cycle of $f(x)$ at level $n\geq 2$. Then, one lift is a weakly splitting cycle at level $n+1$ and the other one weakly grows and then strongly splits.
\end{proposition}

Now, we are ready to prove the Proposition below.

\begin{proposition}\label{prop:Z_2 m=-1(4)}
Let $f(x)=x^m$ over $\Z_2$ and assume $m\geq 2$ with $m\equiv -1~(\mathrm{mod}~4)$. Let $t=\nu_2(m+1)$. For every nonnegative integer $l$, there are exactly two cycles $\{1+2^{l+2}\}$ and $\{-1+2^{l+2}\}$ of length 1 which weakly grow at level $l+3$. All the lifts of $\{1+2^{l+2}\}$ and $\{-1+2^{l+2}\}$ at level $t+l+3$ strongly grow.
\end{proposition}

\begin{proof}
Since $m\equiv -1~(\mathrm{mod}~4)$, we have $t \geq 2$. Write $m=2^t k-1$, then $k$ is odd and $f(x)=x^{2^t k-1}$. We compute that $f(\pm 1)=(\pm 1)^{2^t k-1}=\pm 1$. So, $\pm 1$ are fixed points at any level. Let $l$ be any nonnegative integer and view $\{1\}$ and $\{-1\}$ be cycles of length 1 at level $l+2$.

We compute the quantity $a_{l+2}$, as defined in (\ref{a_l}),
\begin{align*}
a_{l+2}(\pm 1)=f'(\pm 1)=(2^tk-1)(\pm 1)^{2^tk-2}\equiv 3~(\mathrm{mod}~4),
\end{align*}
and the quantity $b_{l+2}$, as defined in (\ref{b_l}),
\begin{align*}
b_{l+2}(\pm 1)=\frac{f(\pm 1)-(\pm 1)}{2^{l+2}}\equiv 0~(\mathrm{mod}~2).
\end{align*}
So, the cycles $\{1\}$ and $\{-1\}$ weakly split at level $l+2$.

The lifts of $\{1\}$ and $\{-1\}$ from level $l+2$ to level $l+3$ are $\{1, 1+2^{l+2}\}$ and $\{-1, -1+2^{l+2}\}$, respectively. The points $1+2^{l+2}$ and $-1+2^{l+2}$ are fixed points at level $l+3$ and by Proposition \ref{ws wg and ss}, weakly grow at level $l+3$. Hence, the lifts at level $l+4$ are $\{1+2^{l+2}, 1+2^{l+2}+2^{l+3}\}$ and $\{-1+2^{l+2}, -1+2^{l+2}+2^{l+3}\}$, respectively, which strongly split.

If we show the following expression, by statement 2 of Proposition \ref{prop:behavior of lifts for p=2}, the proof is complete:
\[ \nu_2\big(b_{l+4}(x)\big)=t-1$ for all $x\in(\pm 1+2^{l+2}+2^{l+4}\Z_2)\cup(\pm 1+2^{l+2}+2^{l+3}+2^{l+4}\Z_2). \]

Suppose that $x=1+2^{l+2}a+2^{l+4}b\in (1+2^{l+2}+2^{l+4}\Z_2)\cup(1+2^{l+2}+2^{l+3}+2^{l+4}\Z_2)$ with $a=1$ or 3 and $b\in\Z_2$. Since $ f^2(x)=x^{2^{t+1}k(2^{t-1}k-1)+1},$ we compute that
\begin{align*}
f^2(1+&2^{l+2}a+2^{l+4}b)-(1+2^{l+2}a+2^{l+4}b)\\
&=(1+2^{l+2}a+2^{l+4}b)^{2^{t+1}k(2^{t-1}k-1)+1}-(1+2^{l+2}a+2^{l+4}b)\\
&=\sum_{j=0}^{2^{t+1}k(2^{t-1}k-1)+1}\tbinom{2^{t+1}k(2^{t-1}k-1)+1}{j}(1+2^{l+2}a)^{2^{t+1}k(2^{t-1}k-1)+1-j}(2^{l+4}b)^j\\
&\quad-(1+2^{l+2}a+2^{l+4}b).
\end{align*}
For an integer $j\geq 2$, the summand has the valuation
\begin{align*}
\nu_2\big(\tbinom{2^{t+1}k(2^{t-1}k-1)+1}{j}&(1+2^{l+2}a)^{2^{t+1}k(2^{t-1}k-1)+1-j}(2^{l+4}b)^j\big)\\
&\geq (t+1)+\frac{j-1}{2}-1+j(l+4)-(j-1)\\
&\geq t+2l+7.
\end{align*}
So, we obtain
\begin{align*}
f^2(1+&2^{l+2}a+2^{l+4}b)-(1+2^{l+2}a+2^{l+4}b)\\
&\equiv(1+2^{l+2}a)^{2^{t+1}k(2^{t-1}k-1)+1}\\
&\quad +\tbinom{2^{t+1}k(2^{t-1}k-1)+1}{1}(1+2^{l+2}a)^{2^{t+1}k(2^{t-1}k-1)}(2^{l+4}b)\\
&\quad-(1+2^{l+2}a+2^{l+4}b)\\
&\equiv ((1+2^{l+2}a)^{2^{t+1}k(2^{t-1}k-1)}-1)(1+2^{l+2}a+2^{l+4}b)\\
&\equiv \Big(\sum_{j=0}^{2^{t+1}k(2^{t-1}k-1)}\tbinom{2^{t+1}k(2^{t-1}k-1)}{j}(2^{l+2}a)^j-1\Big)(1+2^{l+2}a+2^{l+4}b)\\
&\quad (\mathrm{mod}~2^{t+l+4}).
\end{align*}
For an integer $j\geq3$, the summand has the valuation
\begin{align*}
\nu_2\big(\tbinom{2^{t+1}k(2^{t-1}k-1)}{j}(2^{l+2}a)^j\big)&\geq (t+1)+\frac{j-1}{2}-1+j(l+2)-(j-1)\\
&\geq t+3l+5.
\end{align*}
So, we obtain
\begin{align*}
f^2(1+&2^{l+2}a+2^{l+4}b)-(1+2^{l+2}a+2^{l+4}b)\\
&\equiv \big(\tbinom{2^{t+1}k(2^{t-1}k-1)}{1}2^{l+2}a+\tbinom{2^{t+1}k(2^{t-1}k-1)}{2}(2^{l+2}a)^2\big)(1+2^{l+2}a+2^{l+4}b)\\
&\equiv 2^{t+l+3}~(\mathrm{mod}~2^{t+l+4}).
\end{align*}
Hence, $\nu_2\big(b_{l+4}(x)\big)=\nu_2\big(\frac{f^2(1+2^{l+2}a+2^{l+4}b)-(1+2^{l+2}a+2^{l+4}b)}{2^{l+4}}\big)=t-1$.

When $x=-1+2^{l+2}a+2^{l+4}b\in (-1+2^{l+2}+2^{l+4}\Z_2)\cup(-1+2^{l+2}+2^{l+3}+2^{l+4}\Z_2)$ with $a=1$ or 3 and $b\in\Z_2$, we can prove it similarly.
Therefore the proof is complete.
\end{proof}

By Propositions \ref{prop:pZ_p attracting basin}, \ref{prop:Z_2 m=1(4)} and \ref{prop:Z_2 m=-1(4)}, we conclude that the following is true.

\begin{theorem}
Let $f(x)=x^m$ over $\Z_2$ and assume $m\geq 2$ with $m\equiv 1~(\mathrm{mod}~2)$. Let $l$ be a nonnegative integer.
\begin{enumerate}
\item If $m\equiv 1~(\mathrm{mod}~4)$, let $t=\nu_2(m-1)$. Then, the minimal decomposition of $\Z_2$ for $f(x)$ is
\[ \Z_2=\{0,\pm 1\}\bigsqcup \big(\bigcup_{l\geq 0}\bigcup_{a=0}^{2^{t-1}-1}\bigcup_{i=1}^2 M_{l,a,i}\big) \bigsqcup (2\Z_2-\{0\}), \]
where
\begin{align*}
M_{l,a,1}&=1+2^{l+2}+2^{l+3}a+2^{t+l+2}\Z_2~\mathrm{and}\\
M_{l,a,2}&=-1+2^{l+2}+2^{l+3}a+2^{t+l+2}\Z_2.
\end{align*}
Here, $\{0,\pm 1\}$ is the set of fixed points, $M_{l,a,i}$'s are the minimal components, and $2\Z_2-\{0\}$ is the attracting basin of the fixed point 0.

\item If $m\equiv -1~(\mathrm{mod}~4)$, let $t=\nu_2(m+1)$. Then, the minimal decomposition of $\Z_2$ for $f(x)$ is
\[ \Z_2=\{0,\pm 1\}\bigsqcup \big(\bigcup_{l\geq 0}\bigcup_{a=0}^{2^{t-1}-1}\bigcup_{i=1}^2 M_{l,a,i}\big) \bigsqcup (2\Z_2-\{0\}), \]
where
\begin{align*}
M_{l,a,1}&=(1+2^{l+2}+2^{l+4}a+2^{t+l+3}\Z_2)\\
&\quad \cup\big((1+2^{l+2}+2^{l+4}a)^{-1}+2^{t+l+2}+2^{t+l+3}\Z_2\big)~\mathrm{and}\\
M_{l,a,2}&=(-1+2^{l+2}+2^{l+4}a+2^{t+l+3}\Z_2)\\
&\quad \cup\big((-1+2^{l+2}+2^{l+4}a)^{-1}+2^{t+l+2}+2^{t+l+3}\Z_2\big).
\end{align*}
Here, $\{0,\pm 1\}$ is the set of fixed points, $M_{l,a,i}$'s are minimal components, and $2\Z_2-\{0\}$ is the attracting basin of the fixed point 0.
\end{enumerate}
\end{theorem}

\begin{proof}
1. The lifts of $\{1+2^{l+2}\}$ and $\{-1+2^{l+2}\}$ at level $t+l+2$ are of the forms $\{1+2^{l+2}+2^{l+3}a\}$ and $\{-1+2^{l+2}+2^{l+3}a\}$, respectively, where $a\in\{0,\dots,2^{t-1}-1\}$. Hence from Propositions  \ref{prop:pZ_p attracting basin} and \ref{prop:Z_2 m=1(4)}, the statement is proved.

2. The lifts of $\{1+2^{l+2},1+2^{l+2}+2^{l+3}\}$ and $\{-1+2^{l+2},-1+2^{l+2}+2^{l+3}\}$ at level $t+l+3$ are of the forms $\{1+2^{l+2}+2^{l+4}a,(1+2^{l+2}+2^{l+4}a)^{-1}+2^{t+l+2}\}$ and $\{-1+2^{l+2}+2^{l+4}a,(-1+2^{l+2}+2^{l+4}a)^{-1}+2^{t+l+2}\}$, respectively, where $a\in\{0,\dots,2^{t-1}-1\}$. Hence from Propositions  \ref{prop:pZ_p attracting basin} and \ref{prop:Z_2 m=-1(4)}, the statement is proved.
\end{proof}

\section{Minimal decompositions of monomials on $\Z_3$}

In this section, we study the dynamical structure of monomials on the ring $\Z_3$. Let $f(x)=x^m$ with an integer $m\geq 2$. We categories the monomials depending on the exponentials $m$ which we divide into three cases: $m\equiv 0,1$ and $-1~(\mathrm{mod}~3)$.

Firstly, consider the case $m\equiv 0~(\mathrm{mod}~3)$.

\begin{theorem}
Let $f(x)=x^m$ over $\Z_3$ and assume $m\geq 2$ with $m\equiv 0~(\mathrm{mod}~3)$.
\begin{enumerate}
\item If $m\equiv 0~(\mathrm{mod}~6)$, then the minimal decomposition of $\Z_3$ for $f(x)$ is
\[ \Z_3=\{0,1\}\bigsqcup\big((3\Z_3-\{0\})\cup(1+3\Z_3-\{1\})\cup(-1+3\Z_3)\big). \]
Here, $\{0,1\}$ is the set of fixed points and the other sets are the attracting basin of $\{0,1\}$.

\item If $m\equiv 3~(\mathrm{mod}~6)$, then the minimal decomposition of $\Z_3$ for $f(x)$ is
\[ \Z_3=\{0,\pm 1\}\bigsqcup\big((3\Z_3-\{0\})\cup(1+3\Z_3-\{1\})\cup(-1+3\Z_3-\{-1\})\big). \]
Here, $\{0,\pm 1\}$ is the set of fixed points and the other sets are the attracting basin of $\{0,\pm 1\}$.
\end{enumerate}
\end{theorem}

\begin{proof}
1. Assume that $m\equiv 0~(\mathrm{mod}~6)$. Then, $f(1)=1$ and $a_1(1)=f'(1)=m\equiv 0~(\mathrm{mod}~3)$. So, $\{1\}$ is a cycle of length 1 at level 1 which grows tails by Definition \ref{def:movement for p>=3}. By statement 3 of Proposition \ref{prop:behavior of lifts for p>=3}, $f(x)$ has a fixed point 1 in the clopen set $1+3\Z_3$ with $1+3\Z_3$ lying in its attracting basin. We check that $f(-1)\equiv 1~(\mathrm{mod}~3)$ implies $f(-1+3\Z_3)\subseteq 1+3\Z_3$. Therefore, with Proposition \ref{prop:pZ_p attracting basin}, the proof of statement 1 is complete.

2. Assume that $m\equiv 3~(\mathrm{mod}~6)$. Then, $f(\pm 1)=\pm 1$, $a_1(\pm 1)=f'(\pm 1)=m(\pm 1)^{m-1}\equiv 0~(\mathrm{mod}~3)$. So, $\{1\}$ and $\{-1\}$ are cycles of length 1 at level 1 which grow tails by Definition \ref{def:movement for p>=3}. By statement 3 of Proposition \ref{prop:behavior of lifts for p>=3}, $f(x)$ has fixed points $\pm 1$ in the clopen set $1+3\Z_3$ and $-1+3\Z_3$ with $1+3\Z_3$ and $-1+3\Z_3$ lying their attracting basins, respectively. Therefore, with Proposition \ref{prop:pZ_p attracting basin}, the proof of statement 2 is completed.
\end{proof}

Secondly, consider the case $m\equiv 1~(\mathrm{mod}~3)$.

\begin{proposition}\label{prop: p=3, m=1 (mod 3)}
Let $f(x)=x^m$ over $\Z_3$ and assume $m\geq 2$ with $m\equiv 1~(\mathrm{mod}~3)$. Let $t=\nu_3(m-1)$.
\begin{enumerate}
\item If $m\equiv 1~(\mathrm{mod}~6)$, then $f(x)$ has $4\cdot 3^{t-1}$ growing cycles of length 1 at every level $\geq t+1$.

\item If $m\equiv 4~(\mathrm{mod}~6)$, then $f(x)$ has $2\cdot 3^{t-1}$ growing cycles of length 1 at every level $\geq t+1$.
\end{enumerate}
\end{proposition}

\begin{proof}
1. For $m\equiv 1~(\mathrm{mod}~6)$, we have $t \geq 1$. Write $m=3^t k+1$ with $3\not|~k$ and $k$ even. We compute that $f(\pm 1)=\pm 1$. So, $\pm 1$ are fixed points at any level. Let $l$ be any nonnegative integer and view $\{1\}$ and $\{-1\}$ as cycles of length 1 at level $l+1$.

We compute the quantity $a_{l+1}$, as defined in (\ref{a_l}),
\begin{align*}
a_{l+1}(\pm 1)=f'(\pm 1)=(3^t k+1) (\pm 1)^{3^t k} \equiv 1~(\mathrm{mod}~3),
\end{align*}
and the quantity $b_{l+1}$, as defined in (\ref{b_l}),
\begin{align*}
b_{l+1}(\pm 1)=\frac{f(\pm 1)-(\pm 1)}{3^{l+1}}=\frac{(\pm 1)^{3^t k+1}-(\pm 1)}{3^{l+1}}\equiv 0~(\mathrm{mod}~3).
\end{align*}
By Definition \ref{def:movement for p>=3}, the cycles $\{1\}$ and $\{-1\}$ split at level $l+1$ .

The lift of $\{1\}$ from level $l+1$ to level $l+2$ is $\{1, 1+3^{l+1}, 1+2\cdot 3^{l+1}\}$, where each point is fixed by $f$ at level $l+2$. We compute the quantity $A_{l+2}$ for $1+3^{l+1}$ and $1+2\cdot 3^{l+1}$, as defined in (\ref{def: A_l and B_l}). For these points, write $1+3^{l+1} c$ in $\Z_3$ with $c=i+3g$ where $i=1$ or $2$ and for some $g \in \Z_3$. We compute that
\begin{align*}
a_{l+2}(1+3^{l+1}c)&=f'(1+3^{l+1}c)\\
&=(3^t k+1)(1+3^{l+1}c)^{3^t k}\\
&=\sum_{j=0}^{3^t k}\tbinom{3^t k}{j}(3^{l+1}c)^j+\sum_{j=0}^{3^t k}\tbinom{3^t k}{j}3^{j(l+1)+t}c^j k.
\end{align*}
For $j\geq 1$, the summand in the first term has the valuation
\begin{align*}
\nu_3\big(\tbinom{3^t k}{j}(3^{l+1}c)^j\big)&\geq t+\Big\lfloor\frac{j}{3}\Big\rfloor-1+j(l+1)-(j-2)\\
&\geq t+l+1.
\end{align*}
For $j\geq 1$, the summand in the second term has the valuation
\begin{align*}
\nu_3\big(\tbinom{3^t k}{j}3^{j(l+1)+t}c^j k\big)&\geq t+\Big\lfloor\frac{j}{3}\Big\rfloor-1+j(l+1)+t-(j-2)\\
&\geq 2t+l+1.
\end{align*}
So, we obtain that
\begin{align*}
a_{l+2}(1+3^{l+1}c)&\equiv \tbinom{3^t k}{0}+\tbinom{3^t k}{0}3^{t}k\equiv 1+3^t k~(\mathrm{mod}~3^{t+l+1}).
\end{align*}
Hence,
\begin{align*}
A_{l+2}=\nu_3\big(a_{l+2}(1+3^{l+1}c)-1\big)=\nu_3(3^t k)=t.
\end{align*}

Now we compute the quantity $B_{l+2}$ for $1+3^{l+1}$ and $1+2\cdot 3^{l+1}$, as defined in (\ref{def: A_l and B_l}). For these points, as above write $1+3^{l+1} c$ in $\Z_3$ with $c=i+3g$ where $i=1$ or $2$ and for some $g \in \Z_3$. We compute that
\begin{align*}
f(1+3^{l+1}c)-(1+3^{l+1}c)&=(1+3^{l+1}c)^{3^t k+1}-(1+3^{l+1}c)\\
&=\sum_{j=0}^{3^t k+1}\tbinom{3^t k+1}{j}(3^{l+1}c)^j-(1+3^{l+1}c).
\end{align*}
For $j=2$, the summand has the valuation
\begin{align*}
\nu_3\big(\tbinom{3^t k+1}{2}(3^{l+1}c)^2\big)&=\nu_3\big(\frac{1}{2}(3^t k+1)(3^t k)(3^{2(l+1)})c^2\big)=t+2l+2.
\end{align*}
For $j\geq 3$, the summand has the valuation
\begin{align*}
\nu_3\big(\tbinom{3^t k+1}{j}(3^{l+1}c)^j\big)&\geq t+\Big\lfloor\frac{j}{3}\Big\rfloor-1+j(l+1)-(j-2)\\
&\geq t+3l+2.
\end{align*}
So, we obtain that
\begin{align*}
f(1+3^{l+1}c)-(1+3^{l+1}c)&\equiv \tbinom{3^t k+1}{0}+\tbinom{3^t k+1}{1}(3^{l+1}c)-(1+3^{l+1}c)\\
&\equiv 3^{t+l+1}kc~(\mathrm{mod}~3^{t+l+2}).
\end{align*}
Hence,
\begin{align*}
B_{l+2}=\nu_3(b_{l+2})=\nu_3
\Big(\frac{f(1+3^{l+1}c)-(1+3^{l+1}c)}{3^{l+2}}\Big)=t-1.
\end{align*}

If $0\leq l\leq t-3$, we have $A_{l+2}=t$ and $B_{l+2}=t-1$. By Definition (\ref{def: A_l and B_l}) of $B_{l+2}$, $B_{l+2}=t-1$ implies $B_{l+3}=t-2$, $B_{l+4}=t-3$, $\dots$, $B_{l+s+2}=t-s-1$ and $B_{l+s+3}=t-s-2$ for some $\frac{t-l-5}{2}<s\leq \frac{t-l-3}{2}$. Then, $B_{l+s+3}=t-s-2<t=A_{l+s+3}$ and $B_{l+s+3}=t-s-2<l+s+3$. By Proposition \ref{prop: lifts of splitting cycle for p>=3}, all the lifts of $\{1+3^{l+1}i\}$ at level $t+l+1$ grow forever. So, we obtain $2\cdot 3^{t-1}$ growing cycles of length 1 at level $t+l+1$.

If $l>t-3$, since $B_{l+2}=t-1<t=A_{l+2}$ and $B_{l+2}=t-1<l+2$, by Proposition \ref{prop: lifts of splitting cycle for p>=3}, all the lifts of $\{1+3^{l+1} i\}$ at level $t+l+1$ grow forever. So, we obtain $2\cdot 3^{t-1}$ growing cycles of length 1 at level $t+l+1$.

The same conclusion can be drawn for the lifts of $\{-1\}$ by the same argument of the proof for lifts of $\{1\}$. Combine these two, we obtain $4\cdot 3^{t-1}$ growing cycles of length 1 at level $t+l+1$.

2. For $m\equiv 4~(\mathrm{mod}~6)$, we compute that $f(1)=1$ and $f(-1)=1\not\equiv -1~(\mathrm{mod}~3)$. So, only 1 is the fixed point at any level, which is different from statement 1. With the same argument, we obtain that the number of growing cycles at level $\geq t+1$ is $2\cdot 3^{t-1}$, which is a half of that in statement 1. This completes the proof.
\end{proof}

By Propositions \ref{prop:pZ_p attracting basin} and \ref{prop: p=3, m=1 (mod 3)}, we conclude that the following is true.
\begin{theorem}
Let $f(x)=x^m$ over $\Z_3$ and assume $m\geq 2$ with $m\equiv 1~(\mathrm{mod}~3)$. Let $t=\nu_3(m-1)$.
\begin{enumerate}
\item If $m\equiv 1~(\mathrm{mod}~6)$, then the minimal decomposition of $\Z_3$ for $f(x)$ is
\[ \Z_3=\{0,\pm 1\}\bigsqcup \big(\bigcup_{l\geq 0}\bigcup_{a=0}^{3^{t-1}-1}\bigcup_{i=1}^2\bigcup_{j=1}^2 M_{l,a,i,j}\big) \bigsqcup (3\Z_3-\{0\}), \]
where
\begin{align*}
M_{l,a,i,1}&=1+3^{l+1}i+3^{l+2}a+3^{t+l+1}\Z_3~\mathrm{and}\\
M_{l,a,i,2}&=-1+3^{l+1}i+3^{l+2}a+3^{t+l+1}\Z_3.
\end{align*}
Here, $\{0,\pm 1\}$ is the set of fixed points, $M_{l,a,i,j}$'s are the minimal components, and $3\Z_3-\{0\}$ is the attracting basin of the fixed point 0.

\item If $m\equiv 4~(\mathrm{mod}~6)$, then the minimal decomposition of $\Z_3$ for $f(x)$ is
\[ \Z_3=\{0,1\}\bigsqcup \big(\bigcup_{l\geq 0}\bigcup_{a=0}^{3^{t-1}-1}\bigcup_{i=1}^2 M_{l,a,i}\big) \bigsqcup \big((3\Z_3-\{0\})\cup(2+3\Z_3)\big), \]
where $M_{l,a,i}=1+3^{l+1}i+3^{l+2}a+3^{t+l+1}\Z_3$. Here, $\{0,1\}$ is the set of fixed points, $M_{l,a,i}$'s are the minimal components, and $(3\Z_3-\{0\})\cup(2+3\Z_3)$ is the attracting basin.
\end{enumerate}
\end{theorem}

\begin{proof}
1. The lifts of $\{1+3^{l+1}i\}$ and $\{-1+3^{l+1}i\}$ at level $t+l+1$ are of the forms $\{1+3^{l+1}i+3^{l+2}a\}$ and $\{-1+3^{l+1}i+3^{l+2}a\}$, respectively, where $i\in\{1,2\}$ and $a\in\{0,\dots,3^{t-1}-1\}$. Hence from Propositions \ref{prop:pZ_p attracting basin} and \ref{prop: p=3, m=1 (mod 3)}, the statement is proved.

2. The lifts of $\{1+3^{l+1}i\}$ at level $t+l+1$ are of the forms $\{1+3^{l+1}i+3^{l+2}a\}$, respectively, where $i\in\{1,2\}$ and $a\in\{0,\dots,3^{t-1}-1\}$. We check that
$f(2)\equiv 1~(\mathrm{mod}~3)$ implies $f(2+3\Z_3)\subseteq 1+3\Z_3$.
Hence from Propositions \ref{prop:pZ_p attracting basin} and \ref{prop: p=3, m=1 (mod 3)}, the statement is proved.
\end{proof}

Now, we consider the case $m\equiv -1~(\mathrm{mod}~3)$.

\begin{proposition}\label{prop: p=3, m=-1 (mod 3)}
Let $f(x)=x^m$ over $\Z_3$ and assume $m\geq 2$ with $m\equiv -1~(\mathrm{mod}~3)$. Let $t=\nu_3(m+1)$.
\begin{enumerate}
\item If $m\equiv 2~(\mathrm{mod}~6)$, then $f(x)$ has $3^{t-1}$ growing cycles of length 2 at every level $\geq t+1$.

\item If $m\equiv 5~(\mathrm{mod}~6)$, then $f(x)$ has $2\cdot 3^{t-1}$ growing cycles of length 2 at every level $\geq t+1$.
\end{enumerate}
\end{proposition}

\begin{proof}
2. We prove statement 2 firstly. For $m\equiv 5~(\mathrm{mod}~6)$, we have $t\geq 1$. Write $m=3^t k-1$ then $k$ is even. We compute that $f(\pm 1)=\pm 1$. So, $\pm 1$ are fixed points at any level. Let $l$ be any nonnegative integer and view $\{1\}$ and $\{-1\}$ as cycles of length 1 at level $l+1$.

We compute the quantity $a_{l+1}$, as defined in (\ref{a_l}),
\begin{align*}
a_{l+1}(\pm 1)=f'(\pm 1)=(3^t k-1)(\pm 1)^{3^t k-2} \equiv 2~(\mathrm{mod}~3).
\end{align*}
By Definition \ref{def:movement for p>=3}, the cycles $\{1\}$ and $\{-1\}$ partially split at level $l+1$.

By statement 4 in Proposition \ref{prop:behavior of lifts for p>=3}, the lifts of $\{1\}$ and $\{-1\}$ from level $l+1$ to level $l+2$ contain 2-cycles $\{1+3^{l+1},1+3^{l+1}\cdot 2\}$ and $\{-1+3^{l+1},-1+3^{l+1}\cdot 2\}$.

We consider the cycle $\{1+3^{l+1},1+3^{l+1}\cdot 2\}$.  Compute the quantity $a_{l+2}$, as defined in (\ref{a_l}).
\begin{align*}
a_{l+2}(1+3^{l+1})&=f'(1+3^{l+1})\cdot f'(1+3^{l+1}\cdot 2)\\
&=(3^t k-1)(1+3^{l+1})^{3^t k-2}\cdot (3^t k-1)(1+3^{l+1}\cdot 2)^{3^t k-2}\\
&\equiv 1~(\mathrm{mod}~3).
\end{align*}
We compute the quantity $b_{l+2}$, as defined in (\ref{b_l}). Since $(1+3^{l+1})^{3^t k}\equiv 1+3^{t+l+1}k~(\mathrm{mod}~3^{t+l+2})$ and $(1+3^{l+1})^{-1}=1-3^{l+1}+(3^{l+1})^2-(3^{l+1})^3+\dots$,
\begin{align*}
f(1+3^{l+1})&=(1+3^{l+1})^{3^t k-1}\\
&\equiv (1+3^{t+l+1}k)(1+3^{l+1})^{-1}\\
&\equiv (1+3^{l+1})^{-1}+3^{t+l+1}k~(\mathrm{mod}~3^{t+l+2}),
\end{align*}
and then
\begin{align*}
f^2(1+3^{l+1})&-(1+3^{l+1})\\
&\equiv f\big((1+3^{l+1})^{-1}+3^{t+l+1}k\big)-(1+3^{l+1})\\
&\equiv \sum_{j=0}^{3^t k-1}\tbinom{3^t k-1}{j}\big((1+3^{l+1})^{-1}\big)^{3^t k-1-j}(3^{t+l+1}k)^j-(1+3^{l+1})\\
&\equiv \big((1+3^{l+1})^{-1}\big)^{3^t k-1}+\tbinom{3^t k-1}{1}\big((1+3^{l+1})^{-1}\big)^{3^t k-2}(3^{t+l+1}k)\\
&\quad-(1+3^{l+1})-2\cdot 3^{t+l+1} k~(\mathrm{mod}~3^{t+l+2}).
\end{align*}
Since $(1+3^{l+1})^{-3^t k}\equiv 1-3^{t+l+1}k~(\mathrm{mod}~3^{t+l+2})$, then $f^2(1+3^{l+1})-(1+3^{l+1})\equiv -3^{t+l+1}\cdot 2k~(\mathrm{mod}~3^{t+l+2})$. Hence,
\begin{equation*}
f^2(1+3^{l+1})-(1+3^{l+1})=-3^{t+l+1}\cdot 2k+3^{t+l+2}c
\end{equation*}
for some $c\in\Z$. Then,
\begin{equation*}
b_{l+2}(1+3^{l+1})=-2\cdot 3^{t-1} k+3^t c
\begin{cases}
\equiv 0~(\mathrm{mod}~3), &\text{if $t\geq 2$;}\\
\not\equiv 0~(\mathrm{mod}~3), &\text{if $t=1$.}
\end{cases}
\end{equation*}
Therefore the cycle $\{1+3^{l+1},1+3^{l+1}\cdot 2\}$ grows when $t=1$, and splits when $t\geq 2$ at level $l+2$.

Next we compute the quantity $A_{l+2}$, as defined in (\ref{def: A_l and B_l}). For any $g_1,g_2\in\Z_3$,
\begin{align*}
a_{l+2}&(1+3^{l+1}+3^{l+2}g_1)\\
&=f'(1+3^{l+1}+3^{l+2}g_1)\cdot f'(1+3^{l+1}\cdot 2+3^{l+2}g_2)\\
&=(3^t k-1)(1+3^{l+1}+3^{l+2}g_1)^{3^t k-2}\cdot (3^t k-1)(1+3^{l+1}\cdot 2+3^{l+2}g_2)^{3^t k-2}\\
&=(1-3^t\cdot 2k+3^{2t}k^2)(1+3^{l+2}g)^{3^t k-2}
\end{align*}
where $g=1+g_2+g_2+3^l(1+3g_1)(2+3g_2)$. Since $\big(1+3^{l+2}g\big)^{3^t k}\equiv 1~(\mathrm{mod}~3^{t+l+2})$,
\begin{align*}
a_{l+2}(1+3^{l+1}+3^{l+2}g_1)\equiv (1-3^t\cdot 2k+3^{2t}k^2)(1+3^{l+2}g)^{-2}~(\mathrm{mod}~3^{t+l+2}),
\end{align*}
and since $(1+3^{l+2}g)^{-1}=1-(3^{l+2}g)+(3^{l+2}g)^2-(3^{l+2}g)^3+\cdots$,
\begin{align*}
(1+3^{l+2}g)^{-2}&=1-2(3^{l+2}g)+3(3^{l+2}g)^2-4(3^{l+2}g)^3+5(3^{l+2}g)^4-\cdots.
\end{align*}
So, we obtain
\begin{align*}
a_{l+2}&(1+3^{l+1}+3^{l+2}g_1)\\
&\equiv (1+3^{l+2}g)^{-2}-3^t\cdot 2k+3^{2t}k^2\\
&\equiv 1-2(3^{l+2}g)+3(3^{l+2}g)^2-4(3^{l+2}g)^3+5(3^{l+2}g)^4-\dots\\
&\quad -3^t\cdot 2k+3^{2t}k^2~(\mathrm{mod}~3^{t+l+2}).
\end{align*}
Hence,
\begin{equation*}
A_{l+2}=\nu_3\big(a_{l+2}(1+3^{l+1}+3^{l+2}g_1)-1\big)
\begin{cases}
 \geq l+2 &\text{if}\ l+2<t,\\
 \geq t &\text{if}\ l+2=t,\\
 = t &\text{if}\ l+2>t.
\end{cases}
\end{equation*}

We compute the quantity $B_{l+2}$, as defined in (\ref{def: A_l and B_l}). Let $h=1+3g_1$. Then, $b_{l+2}(1+3^{l+1}+3^{l+2}g_1)=b_{l+2}(1+3^{l+1}h)$. Since $(1+3^{l+1}h)^{3^t k}\equiv 1+3^{t+l+1}kh~(\mathrm{mod}~3^{t+l+2})$ and $(1+3^{l+1}h)^{-1}=1-3^{l+1}h+(3^{l+1}h)^2-(3^{l+1}h)^3+\cdots$, we obtain that
\begin{align*}
f(1+3^{l+1}h)&=(1+3^{l+1}h)^{3^t k-1}\\
&\equiv (1+3^{l+1}h)^{-1}+3^{t+l+1}kh~(\mathrm{mod}~3^{t+l+2}).
\end{align*}
So we compute the following.
\begin{align*}
f^2(1+&3^{l+1}h)-(1+3^{l+1}h)\\
&\equiv f\big((1+3^{l+1}h)^{-1}+3^{t+l+1}kh\big)-(1+3^{l+1}h)\\
&\equiv \big((1+3^{l+1}h)^{-1}+3^{t+l+1}kh\big)^{3^t k-1}-(1+3^{l+1}h)\\
&\equiv \sum_{j=0}^{3^t k-1}\tbinom{3^t k-1}{j}\big((1+3^{l+1}h)^{-1}\big)^{3^t k-1-j}(3^{t+l+1}kh)^{j}-(1+3^{l+1}h)\\
&\equiv \tbinom{3^t k-1}{0}\big((1+3^{l+1}h)^{-1}\big)^{3^t k-1}+\tbinom{3^t k-1}{1}\big((1+3^{l+1}h)^{-1}\big)^{3^t k-2}(3^{t+l+1}kh)\\
&\quad -(1+3^{l+1}h)\\
&\equiv (1+3^{l+1}h)^{-3^t k+1}-(1+3^{l+1}h)^{-3^t k+2}(3^{t+l+1}kh)\\
&\quad -(1+3^{l+1}h)~(\mathrm{mod}~3^{t+l+2}).
\end{align*}
Since $(1+3^{l+1}h)^{-3^t k}\equiv 1-3^{t+l+1}kh~(\mathrm{mod}~3^{t+l+2})$,
\begin{align*}
f^2(1+3^{l+1}h)-(1+3^{l+1}h)&\equiv 1+3^{l+1}h-3^{t+l+1}\cdot 2kh-(1+3^{l+1}h)\\
&\equiv -3^{t+l+1}\cdot 2kh~(\mathrm{mod}~3^{t+l+2}).
\end{align*}
Therefore, we obtain
\begin{align*}
B_{l+2}=\nu_3(b_{l+2}(1+3^{l+1}h))=\nu_3\Big(\frac{f^2(1+3^{l+1}h)-(1+3^{l+1}h)}{3^{l+2}}\Big)=t-1.
\end{align*}

If $l+2<t$, then we have $A_{l+2}\geq l+2$ and $B_{l+2}=t-1$. By Definition (\ref{def: A_l and B_l}), we can check that $A_n\geq l+2$ for any level $n\geq l+2$. So, $B_t=B_{l+2+(t-l-2)}=t-1-(t-l-2)=l+1<l+2\leq A_t$ and $B_t=l+1<l+2<t$. Therefore, by Proposition \ref{prop: lifts of splitting cycle for p>=3}, all the lifts of $\{1+3^{l+1},1+3^{l+1}\cdot 2\}$ at level $t+l+1$ grow forever, hence we obtain that $3^{t-l-2}\cdot 3^{l+1}=3^{t-1}$ growing cycles of length 2 at level $t+l+1$.

If $l+2\geq t$, then we have $A_{l+2}\geq t$ and $B_{l+2}=t-1$. So, $B_{l+2}=t-1<t\leq A_{l+2}$ and $B_{l+2}=t-1<t\leq l+2$. Therefore, by Proposition \ref{prop: lifts of splitting cycle for p>=3}, all the lifts of $\{1+3^{l+1},1+3^{l+1}\cdot 2\}$ at level $t+l+1$ grow forever, hence we obtain that $3^{t-1}$ growing cycles of length 2 at level $t+l+1$.

For the cycle $\{-1+3^{l+1}, -1+3^{l+1}\cdot 2\}$, we can analysis analogously. So, we obtain that $3^{t-1}$ growing cycles of length 2 at level $t+l+1$. Combine these two cases, we conclude that there are $2\cdot 3^{t-1}$ growing cycles of length 2 at level $t+l+1$, as required.

1. For $m\equiv 2~(\mathrm{mod}~6)$, we compute that $f(1)=1$ and $f(-1)=1\not\equiv -1~(\mathrm{mod}~3)$. So, only 1 is a fixed point at any level, which is different from the case of statement 2. With the same argument, we obtain that the number of growing cycles at level $\geq t+1$, which is a half of that in statement 2. This completes the proof.
\end{proof}

By Propositions \ref{prop:pZ_p attracting basin} and \ref{prop: p=3, m=-1 (mod 3)}, we conclude that the following is true.
\begin{theorem}
Let $f(x)=x^m$ over $\Z_3$ and assume $m\geq 2$ with $m\equiv 2~(\mathrm{mod}~3)$. Let $t=\nu_3(m+1)$.
\begin{enumerate}
\item If $m\equiv 2~(\mathrm{mod}~6)$, then the minimal decomposition of $\Z_3$ for $f(x)$ is
\[ \Z_3=\{0,1\}\bigsqcup \big(\bigcup_{l\geq 0}\bigcup_{a=0}^{3^{t-1}-1}\bigcup_{i=1}^2 M_{l,a,i}\big) \bigsqcup \big((3\Z_3-\{0\})\cup(2+3\Z_3)\big), \]
where $M_{l,a}=(1+3^{l+1}+3^{l+2}a+3^{t+l+1}\Z_3)\cup \big((1+3^{l+1}+3^{l+2}a)^{-1}+3^{t+l+1}\Z_3\big)$. Here, $\{0,1\}$ is the set of fixed points, $M_{l,a}$'s are the minimal components, and $3\Z_3-\{0\}$ and $2+3\Z_3$ are the attracting basin of 0 and $1+3\Z_3$, respectively.

\item If $m\equiv 5~(\mathrm{mod}~6)$, then the minimal decomposition of $\Z_3$ for $f(x)$ is
\[ \Z_3=\{0,\pm 1\}\bigsqcup \big(\bigcup_{l\geq 0}\bigcup_{a=0}^{3^{t-1}-1}\bigcup_{i=1}^2 M_{l,a,i}\big) \bigsqcup (3\Z_3-\{0\}), \]
where
\begin{align*}
M_{l,a,1}&=(1+3^{l+1}+3^{l+2}a+3^{t+l+1}\Z_3)\\
&\quad \cup \big((1+3^{l+1}+3^{l+2}a)^{-1}+3^{t+l+1}\Z_3\big)~\mathrm{and}\\
M_{l,a,2}&=(-1+3^{l+1}+3^{l+2}a+3^{t+l+1}\Z_3)\\
&\quad \cup \big((-1+3^{l+1}+3^{l+2}a)^{-1}+3^{t+l+1}\Z_3\big).
\end{align*}
Here, $\{0,\pm 1\}$ is the set of fixed points, $M_{l,a,i}$'s are the minimal components, and $3\Z_3-\{0\}$ is the attracting basin of the fixed point 0.
\end{enumerate}
\end{theorem}

\begin{proof}
1. The lifts of $\{1+3^{l+1},1+3^{l+1}\cdot 2\}$ at level $t+l+1$ are of the forms $\{1+3^{l+1}+3^{l+2}a,(1+3^{l+1}+3^{l+2}a)^{-1}\}$, where $a\in\{0,\dots,3^{t-1}-1\}$. We check that $f(2)\equiv 1~(\mathrm{mod}~3)$ implies $f(2+3\Z_3)\subseteq 1+3\Z_3$. Hence from Propositions \ref{prop:pZ_p attracting basin} and \ref{prop: p=3, m=-1 (mod 3)}, the statement is proved.

2. The lifts of $\{1+3^{l+1},1+3^{l+1}\cdot 2\}$ and $\{-1+3^{l+1},-1+3^{l+1}\cdot 2\}$ at level $t+l+1$ are of the forms $\{1+3^{l+1}+3^{l+2}a,(1+3^{l+1}+3^{l+2}a)^{-1}\}$ and $\{-1+3^{l+1}+3^{l+2}a,(-1+3^{l+1}+3^{l+2}a)^{-1}\}$, respectively, where $a\in\{0,\dots,3^{t-1}-1\}$. Hence from Propositions \ref{prop:pZ_p attracting basin} and \ref{prop: p=3, m=-1 (mod 3)}, the statement is proved.
\end{proof}


\section{Minimal decompositions of monomials on $\Z_5$}

In this section, we study the dynamical structure of monomials on the ring $\Z_5$. Let $f(x)=x^m$ with an integer $m\geq 2$. We categories the monomials depending on the exponentials $m$ which we divide into five cases: $m\equiv 0,1,-1,2,$ and $-2~(\mathrm{mod}~5)$.

The polynomial $x^5-x$ over $\Z_5$ is factored as $x(x-1)(x+1)(x-\bi)(x+\bi)$ where $\bi=2+5+2\cdot 5^2+\cdots \in \Z_5$ and $-\bi=3+3\cdot 5+2\cdot 5^2+\cdots \in \Z_5$. We say that $\bi\equiv 2$ and $-\bi\equiv 3~(\mathrm{mod}~5)$.

Firstly, consider the case $m\equiv 0 ~(\mathrm{mod}~5)$.

\begin{theorem}
Let $f(x)=x^m$ over $\Z_5$ and assume $m\geq 2$ with $m\equiv 0~(\mathrm{mod}~5)$.
\begin{enumerate}
\item If $m\equiv 0~(\mathrm{mod}~10)$, then the minimal decomposition of $\Z_5$ for $f(x)$ is
\begin{align*}
\Z_5&=\{0,1\}\bigsqcup\big((5\Z_5-\{0\})\cup(1+5\Z_5-\{1\})\cup(2+5\Z_5)\cup(3+5\Z_5)\\
&\quad \cup(4+5\Z_5)\big).
\end{align*}
Here, $\{0,1\}$ is the set of fixed points and the sets $5\Z_5-\{0\}$ and $(1+5\Z_5-\{1\})\cup(2+5\Z_5)\cup(3+5\Z_5)\cup(4+5\Z_5)$ are the attracting basin of the fixed points $0$ and $1$, respectively.

\item If $m\equiv 5~(\mathrm{mod}~20)$, then the minimal decomposition of $\Z_5$ for $f(x)$ is
 \begin{align*}
 \Z_5&=\{0,\pm 1,\pm\bi\}\bigsqcup\big((5\Z_5-\{0\})\cup(1+5\Z_5-\{1\})\\
 &\quad \cup(2+5\Z_5-\{\bi\})\cup(3+5\Z_5-\{-\bi\})\cup(4+5\Z_5-\{-1\})\big).
 \end{align*}
 Here, $\{0,\pm 1,\pm\bi\}$ is the set of fixed points and the sets
 $5\Z_5-\{0\}$, $1+5\Z_5-\{1\}$, $2+5\Z_5-\{\bi\}$, $3+5\Z_5-\{-\bi\}$, and $4+5\Z_5-\{-1\}$ are the attracting basin of the fixed points $0, 1, \bi$, $-\bi$ and $-1$, respectively.

\item If $m\equiv 15~(\mathrm{mod}~20)$, then the minimal decomposition of $\Z_5$ for $f(x)$ is
 \begin{align*}
 \Z_5&=\{0,\pm 1,\pm\bi\}\bigsqcup\big((5\Z_5-\{0\})\cup(1+5\Z_5-\{1\})\\
 &\quad \cup(2+5\Z_5-\{\bi\})\cup(3+5\Z_5-\{-\bi\})\cup(4+5\Z_5-\{-1\})\big).
 \end{align*}
 Here, $\{0,\pm 1,\pm\bi\}$ is the set of periodic points, where $0,\pm 1$ are fixed points and $\pm\bi$ are two periodic points to each other. The sets $5\Z_5-\{0\}$, $1+5\Z_5-\{1\}$, $4+5\Z_5-\{-1\}$ and $(2+5\Z_5-\{\bi\})\cup(3+5\Z_5-\{-\bi\})$ are the attracting basin of $0, 1, -1$ and $\{\pm\bi\}$, respectively.
\end{enumerate}
\end{theorem}

\begin{proof}
1. For $m\equiv 0~(\mathrm{mod}~10)$, write $m=10 d$ with $d\geq 1$. Since $f(1)=1$, $\{1\}$ is a cycle of length 1 at level 1.

We compute the quantity $a_1$ for the above cycles, as defined in (\ref{a_l}). Since $f'(x)=10d\cdot x^{10d-1}$, then $a(1)=f'(1)=10d\equiv 0~(\mathrm{mod}~5)$. Therefore, $\{1\}$ grows tails at level 1 by Definition \ref{def:movement for p>=3}. By statement 3 of Proposition \ref{prop:behavior of lifts for p>=3}, $f(x)$ has a fixed point $1\in 1+5\Z_5$, with $1+5\Z_5$ lying in its attracting basin.

If $d$ is even, then $f(2)\equiv 1$, $f(3)\equiv 1$ and $f(4)\equiv 1$ (mod 5), which imply that $f(2+5\Z_5)\cup f(3+5\Z_5)\cup f(4+5\Z_5)\subset 1+5\Z_5$. So, $(2+5\Z_5)\cup (3+5\Z_5)\cup (4+5\Z_5)$ lies in the attracting basin of 1.

If $d$ is odd, then $f(2)\equiv 4$, $f(3)\equiv 4$ and $f(4)\equiv 1$ (mod 5), which imply that $f(2+5\Z_5)\cup f(3+5\Z_5)\subset 4+5\Z_5$ and $f(4+5\Z_5)\subset 1+5\Z_5$. So, $(2+5\Z_5)\cup (3+5\Z_5)\cup (4+5\Z_5)$ lies in the attracting basin of 1.

Therefore, with Proposition \ref{prop:pZ_p attracting basin}, the proof of statement 1 is completed.

2. For $m\equiv 5~(\mathrm{mod}~20)$, write $m=5+20d$ with $d\geq 0$. Since $f(\pm 1)=\pm 1$ and $f(\pm\bi)=\pm\bi$, $\{1\}$, $\{-1\}$, $\{\bi\}$ and $\{-\bi\}$ are cycles of length 1 at level 1.

We compute the quantity $a_1$ for the above cycles, as defined in (\ref{a_l}). Since $f'(x)=(5+20d)x^{4+20d}$, then $a(\pm 1)=f'(\pm 1)=(5+20d)(\pm 1)^{4+20d}\equiv 0$ and $a(\pm\bi)=f'(\pm\bi)=(5+20d)(\pm\bi)^{4+20d}\equiv 0~(\mathrm{mod}~5)$. Therefore, $\{1\}$, $\{-1\}$, $\{\bi\}$ and $\{-\bi\}$ grow tails at level 1 by Definition \ref{def:movement for p>=3}. The monomial $f(x)$ has fixed points $1$ in $1+5\Z_5$, $\bi$ in $2+5\Z_5$, $-\bi$ in $3+5\Z_5$ and $-1$ in $4+5\Z_5$, hence by statement 3 of Proposition \ref{prop:behavior of lifts for p>=3} and Proposition \ref{prop:pZ_p attracting basin}, the proof of statement 2 is completed.

3. For $m\equiv 15~(\mathrm{mod}~20)$, write $m=15+20 d$ with $d\geq 0$. Since $f(\pm 1)=\pm 1$ and $f(\pm\bi)=\mp\bi$, $\{1\}$ and $\{-1\}$ are cycles of length 1 and $\{\pm\bi\}$ is a cycle of length 2 at level 1.

We compute the quantity $a_1$ for the above cycles, as defined in (\ref{a_l}). Since $f'(x)=(15+20d)x^{14+20d}$, then $a(\pm 1)=f'(\pm 1)=(15+20d)(\pm 1)^{14+20d}\equiv 0$ and $a(\bi)=f'(\bi)f'(-\bi)=(15+20d)\bi^{14+20d}\cdot (15+20d)(-\bi)^{14+20d}\equiv 0~(\mathrm{mod}~5)$. Therefore, $\{1\}$, $\{-1\}$ and $\{\pm \bi\}$ grow tails at level 1 by Definition \ref{def:movement for p>=3}. The monomial $f(x)$ has periodic points $1$ in $1+5\Z_5$, $-1$ in $4+5\Z_5$, $\pm\bi$ in $(2+5\Z_5)\cup (3+5\Z_5)$, hence by statement 3 of Proposition \ref{prop:behavior of lifts for p>=3} and Proposition \ref{prop:pZ_p attracting basin}, the proof of statement 3 is completed.
\end{proof}

Now we consider the case $m\equiv 1 ~(\mathrm{mod}~5)$.

\begin{proposition}\label{prop: p=5, m=1 (mod 5)}
Let $f(x)=x^m$ over $\Z_5$ and assume $m\geq 2$ with $m\equiv 1~(\mathrm{mod}~5)$. Let $t=\nu_5(m-1)$.
\begin{enumerate}
\item If $m\equiv 1~(\mathrm{mod}~20)$, then $f(x)$ has $16\cdot 5^{t-1}$ growing cycles of length 1 at every level $\geq t+1$.

\item If $m\equiv 6~(\mathrm{mod}~10)$, then $f(x)$ has $4\cdot 5^{t-1}$ growing cycles of length 1 at every level $\geq t+1$.

\item If $m\equiv 11~(\mathrm{mod}~20)$, then $f(x)$ has $8\cdot 5^{t-1}$ growing cycles of length 1 and $4\cdot 5^{t-1}$ growing cycles of length 2 at every level $\geq t+1$.
\end{enumerate}
\end{proposition}

\begin{proof}
1. For $m\equiv 1~(\mathrm{mod}~20)$, we have $t\geq 1$. Write $m=5^t\cdot 4d+1$ for some $d$. We compute that $f(\pm 1)=\pm1$ and $f(\pm\bi)=(\pm\bi)^{5^t\cdot 4d+1}=\pm\bi$. So, $\pm 1$ and $\pm\bi$ are fixed points at any level. Let $l$ be any nonnegative integer and view $\{1\}$, $\{-1\}$, $\{\bi\}$ and $\{\bi\}$ as cycles of length 1 at level $l+1$.

We compute the quantity $a_{l+1}$, as defined in (\ref{a_l}),
\begin{align*}
&a_{l+1}(\pm 1)=f'(\pm 1)=(5^t\cdot 4d+1)(\pm 1)^{5^t\cdot 4d} \equiv 1~(\mathrm{mod}~5)~\mathrm{and}\\
&a_{l+1}(\pm\bi)=f'(\pm\bi)=(5^t\cdot 4d+1)(\pm\bi)^{5^t\cdot 4d} \equiv 1~(\mathrm{mod}~5),
\end{align*}
and the quantity $b_{l+1}$, as defined in (\ref{b_l}),
\begin{align*}
&f(\pm 1)-(\pm 1)=(\pm 1)^{5^t\cdot 4d+1}-(\pm 1)=\pm 1-(\pm 1)=0~\mathrm{and}\\
&f(\pm\bi)-(\pm\bi)=(\pm\bi)^{5^t\cdot 4d+1}-(\pm\bi)=\pm\bi-(\pm\bi)=0,
\end{align*}
hence
\begin{align*}
b_{l+1}(\pm 1)=\frac{f(\pm 1)-(\pm 1)}{5^{l+1}}=0~\mathrm{and}~b_{l+1}(\pm\bi)=\frac{f(\pm\bi)-(\pm\bi)}{5^{l+1}}=0.
\end{align*}
Therefore the cycles $\{1\}$, $\{-1\}$, $\{\bi\}$ and $\{-\bi\}$ split at level $l+1$.

For the cycles $\{1\}$ and $\{-1\}$, with the similar process of the proof of Proposition \ref{prop: p=3, m=1 (mod 3)}, we obtain $4\cdot 5^{t-1}$ growing lifts of length 1 at level $t+l+1$, respectively.

The lift of $\{\bi\}$ from level  $l+1$ to level $l+2$ is $\{\bi+5^{l+1}i\}$ with $i=0, \dots, 4$, where each point is fixed by $f$ at level $l+2$. We compute the quantity $A_{l+2}$ for the points $\bi+5^{l+1}i $ with $i=1, \dots, 4$, as defined in (\ref{def: A_l and B_l}). For these points, write of the form $\bi+5^{l+1} c$ in $\Z_5$ with $c=i+5g$ where $i=1,2,3$, or $4$ and for some $g\in \Z_5$. We compute that
\begin{align*}
a_{l+2}(\bi+5^{l+1}c)&=f'(\bi+5^{l+1}c)\\
&=(5^t\cdot 4d+1)(\bi+5^{l+1}c)^{5^t\cdot 4d}\\
&=\sum_{j=0}^{5^t\cdot 4d}\tbinom{5^t\cdot 4d}{j}\bi^{5^t\cdot 4d-j}5^{j(l+1)+t}c^j\cdot 4d+\sum_{j=0}^{5^t\cdot 4d}\tbinom{5^t\cdot 4d}{j}\bi^{5^t\cdot 4d-j}(5^{l+1}c)^j,
\end{align*}
For $j\geq 1$. the summand in the first term has the valuation
\begin{align*}
\nu_5\big(\tbinom{5^t\cdot 4d}{j}\bi^{5^t\cdot 4d-j}5^{j(l+1)+t}c^j\cdot 4d\big)&\geq t+\Big\lfloor\frac{j}{5}\Big\rfloor-1+j(l+1)+t-(j-4)\\
&\geq 2t+l+3,
\end{align*}
and the summand in the second term has the valuation
\begin{align*}
\nu_5\big(\tbinom{5^t\cdot 4d}{j}\bi^{5^t\cdot 4d-j}(5^{l+1}c)^j\big)&\geq t+\Big\lfloor\frac{j}{5}\Big\rfloor-1+j(l+1)-(j-4)\\
&\geq t+l+3.
\end{align*}
Therefore we obtain
\begin{align*}
a_{l+2}(\bi+5^{l+1}c)\equiv \tbinom{5^t\cdot 4d}{0}\bi^{5^t\cdot 4d}5^{t}\cdot 4d+\tbinom{5^t\cdot 4d}{0}\bi^{5^t\cdot 4d}\equiv 1+5^t\cdot 4d~(\mathrm{mod}~5^{t+l+2}),
\end{align*}
and then
\begin{align*}
A_{l+2}=\nu_5\big(a_{l+2}(\bi+5^{l+1}c)-1\big)=\nu_5(5^t\cdot 4d)=t.
\end{align*}

We compute the quantity $B_{l+2}$, as defined in (\ref{def: A_l and B_l}). We obtain that
\begin{align*}
f(\bi+5^{l+1}c)-(\bi+5^{l+1}c)&=(\bi+5^{l+1}c)^{5^t\cdot 4d+1}-(\bi+5^{l+1}c)\\
&=\sum_{j=0}^{5^t\cdot 4d+1}\tbinom{5^t\cdot 4d+1}{j}\bi^{5^t\cdot 4d+1-j}(5^{l+1}c)^j-(\bi+5^{l+1}c).
\end{align*}
For $j\geq 2$, the summand has the valuation
\begin{align*}
\nu_5\big(\tbinom{5^t\cdot 4d+1}{j}\bi^{5^t\cdot 4d+1-j}(5^{l+1}c)^j\big)&\geq t+\Big\lfloor\frac{j}{5}\Big\rfloor-1+j(l+1)-(j-4)\\
&\geq t+2l+3.
\end{align*}
Therefore, we obtain
\begin{align*}
f(\bi+5^{l+1}&c)-(\bi+5^{l+1}c)\\
&\equiv \tbinom{5^t\cdot 4d+1}{0}\bi^{5^t\cdot 4d+1}+\tbinom{5^t\cdot 4d+1}{1}\bi^{5^t\cdot 4d}(5^{l+1}c)-(\bi+5^{l+1}c)\\
&\equiv 5^{t+l+1}\cdot 4cd~(\mathrm{mod}~5^{t+l+2}),
\end{align*}
and then
\begin{align*}
B_{l+2}=\nu_5(b_{l+2})=\nu_5\big(\frac{f(\bi+5^{l+1}c)-(\bi+5^{l+1}c)}{5^{l+2}}\big)=t-1.
\end{align*}

If $0\leq l\leq t-3$, we obtain $A_{l+2}=t$ and $B_{l+2}=t-1$. By Definition (\ref{def: A_l and B_l}) of $B_{l+2}$, $B_{l+2}=t-1$ implies $B_{l+3}=t-2$, $B_{l+4}=t-3$, $\dots$, $B_{l+s+2}=t-s-1$ and $B_{l+s+3}=t-s-2$ for some $\frac{t-l-5}{2}<s\leq \frac{t-l-3}{2}$. Then, $B_{l+s+3}=t-s-2<t=A_{l+s+3}$ and $B_{l+s+3}=t-s-2<l+s+3$. By Proposition \ref{prop: lifts of splitting cycle for p>=3}, all the lifts of $\{\bi+5^{l+1}i\}$ at level $t+l+1$ grow forever. So, we obtain $4\cdot 5^{t-1}$ growing cycles of length 1 at level $t+l+1$.

For $l>t-3$, since $B_{l+2}=t-1<t=A_{l+2}$ and $B_{l+2}=t-1<l+2$, by Proposition \ref{prop: lifts of splitting cycle for p>=3}, all the lifts of $\{\bi+5^{l+1}i\}$ at level $t+l+1$ grow forever. So, we obtain $4\cdot 5^{t-1}$ growing cycles of length 1 at level $t+l+1$.

The same conclusion can be drawn for the lifts of $\{-\bi\}$ by the same argument of the proof for lifts of $\{\bi\}$. Combine these two, we obtain $16\cdot 5^{t-1}$ growing cycles of length 1 at every level $\geq t+1$.

2. For $m\equiv 6~(\mathrm{mod}~10)$, we compute that $f(1)=1$ and $f(j) \not\equiv j~(\mathrm{mod}~5)$ for $j=2,3,$ or $4$. So, only 1 is the fixed point at any level, which is different from statement 1. With the same argument, we obtain that the number of growing cycles at every level $\geq t+1$ is $4\cdot 5^{t-1}$, which is a quarter of that in statement 1. This completes the proof.

3. For $m\equiv 11~(\mathrm{mod}~20)$, we have $t\geq 1$. Write $m=5^t(2+4d)+1$ for some $d$. We compute that $f(\pm 1)=\pm1$ and $f(\pm\bi)=(\pm\bi)^{5^t(2+4d)+1}=\mp\bi$. So, $\pm 1$ are fixed points and $\{\bi,-\bi\}$ is a cycle of length 2 at any level. Let $l$ be any nonnegative integer and view  $\{1\}$, $\{-1\}$ and $\{\bi,-\bi\}$ as cycles at level $l+1$.

We compute the quantity $a_{l+1}$, as defined in (\ref{a_l}),
\begin{align*}
a_{l+1}(\pm 1)&=f'(\pm 1)=\big(5^t(2+4d)+1\big)(\pm 1)^{5^t(2+4d)} \equiv 1~(\mathrm{mod}~5),\\
a_{l+1}(\bi)&=f'(\bi)f'(-\bi)\\
&=\big(5^t(2+4d)+1\big)\bi^{5^t(2+4d)}\cdot \big(5^t(2+4d)+1\big)(-\bi)^{5^t(2+4d)}\\
&\equiv 1~(\mathrm{mod}~5),
\end{align*}
and the quantity $b_{l+1}$, as defined in (\ref{b_l}),
\begin{align*}
f(\pm 1)-(\pm 1)&=(\pm 1)^{5^t(2+4d)+1}-(\pm 1)=\pm 1-(\pm 1)=0,\\
f^2(\bi)-(\bi)&=f(-\bi)-(\bi)=\bi-\bi=0,
\end{align*}
and so
\begin{align*}
b_{l+1}(\pm 1)=\frac{f(\pm 1)-(\pm 1)}{5^{l+1}}=0~\mathrm{and}~b_{l+1}(\pm\bi)=\frac{f^2(\bi)-(\bi)}{5^{l+1}}=0.
\end{align*}
Hence, the cycles $\{1\}$, $\{-1\}$ and $\{\bi,-\bi\}$ split at level $l+1$.

From the cycles $\{1\}$ and $\{-1\}$, similarly with statement 1, we obtain $8\cdot 5^{t-1}$ growing cycles of length 1 at level $t+l+1$.

The lift of $\{\bi,-\bi\}$ from level $l+1$ to level $l+2$ is
$\cup_{i=0}^4 \{\bi+5^{l+1}i,-\bi-5^{l+1}i\}$,
where $\{\bi+5^{l+1}i,-\bi-5^{l+1}i\}$ is a cycle of length 2 by $f$ at level $l+2$. We compute the quantity $A_{l+2}$ for cycles $i\in\{1,2,3,4\}$, as defined in (\ref{def: A_l and B_l}).
\begin{align*}
a_{l+2}(&\bi+5^{l+1}i+5^{l+2}g_1)\\
&=f'(\bi+5^{l+1}i+5^{l+2}g_1)f'(-\bi-5^{l+1}i+5^{l+2}g_2)\\
&=\big(5^t(2+4d)+1\big)(\bi+5^{l+1}i+5^{l+2}g_1)^{5^t(2+4d)}\\
&\quad \times\big(5^t(2+4d)+1\big)(-\bi-5^{l+1}i+5^{l+2}g_2)^{5^t(2+4d)}\\
&=\big(1+5^t(4+8d)+5^{2t}(2+4d)^2\big)(1+5^{l+1}g)^{5^t(4+8d)}
\end{align*}
where $g_1, g_2 \in \Z_5$ and $g=-2\bi-5^{l+1}i^2+5(-g_1+g_2)(\bi+5^{l+1}i)+5^{l+3}g_1g_2$. Since $(1+5^{l+1}g)^{5^t(4+8d)}\equiv 1+5^{t+l+1}(4+8d)g~(\mathrm{mod}~5^{t+l+2})$,
\begin{align*}
a_{l+2}(&\bi+5^{l+1}i+5^{l+2}g_1)\\
&\equiv \big(1+5^t(4+8d)+5^{2t}(2+4d)^2\big)\big(1+5^{t+l+1}(4+8d)g\big)\\
&\equiv 1+5^t(4+8d)+5^{2t}(2+4d)^2+5^{t+l+1}(4+8d)g~(\mathrm{mod}~5^{t+l+2}).
\end{align*}
Therefore,
\begin{align*}
A_{l+2}=\nu_5\big(a_{l+2}(\bi+5^{l+1}i+5^{l+2}g)-1\big)=t.
\end{align*}

We compute the quantity $B_{l+2}$, as defined in (\ref{def: A_l and B_l}). Let $h=i+5g_1$, for $ \in \Z_5$.
\begin{align*}
f(\bi+5^{l+1}h)&=(\bi+5^{l+1}h)^{5^t(2+4d)+1}\\
&=\sum_{j=0}^{5^t(2+4d)+1}\tbinom{5^t(2+4d)+1}{j}\bi^{5^t(2+4d)+1-j}(5^{l+1}h)^j.
\end{align*}
For $j\geq 2$, the summand has the valuation
\begin{align*}
\nu_5\big(\tbinom{5^t(2+4d)+1}{j}\bi^{5^t(2+4d)+1}(5^{l+1}h)^j\big)&\geq t+\Big\lfloor\frac{j}{5}\Big\rfloor-1+j(l+1)-(j-4)\\
&\geq t+2l+3.
\end{align*}
So, we obtain
\begin{align*}
f(\bi+&5^{l+1}h)\\
&\equiv \tbinom{5^t(2+4d)+1}{0}\bi^{5^t(2+4d)+1}+\tbinom{5^t(2+4d)+1}{1}\bi^{5^t(2+4d)}(5^{l+1}h)\\
&\equiv -\bi-\big(5^t(2+4d)+1\big)5^{l+1}h~(\mathrm{mod}~5^{t+l+2}).
\end{align*}
Now we compute the following.
\begin{align*}
&f^2(\bi+5^{l+1}h)-(\bi+5^{l+1}h)\\
&\equiv -f\Big(\bi+\big(5^t(2+4d)+1\big)5^{l+1}h\Big)-(\bi+5^{l+1}h)\\
&\equiv -\sum_{j=0}^{5^t(2+4d)+1}\tbinom{5^t(2+4d)+1}{j}\bi^{5^t(2+4d)+1-j}\Big(\big(5^t(2+4d)+1\big)5^{l+1}h\Big)^j\\
&\quad -(\bi+5^{l+1}h)~(\mathrm{mod}~5^{t+l+2}).
\end{align*}
For $j\geq 2$, the summand has the valuation
\begin{align*}
\nu_5\bigg(\tbinom{5^t(2+4d)+1}{j}\bi^{5^t(2+4d)+1-j}&\Big(\big(5^t(2+4d)+1\big)5^{l+1}h\Big)^j\bigg)\\
&\geq t+\Big\lfloor\frac{j}{5}\Big\rfloor-1+j(l+1)-(j-4)\\
&\geq t+2l+3.
\end{align*}
So, we obtain
\begin{align*}
f^2(\bi+5^{l+1}h)&-(\bi+5^{l+1}h)\\
&\equiv -\tbinom{5^t(2+4d)+1}{0}\bi^{5^t(2+4d)+1}\\
&\quad -\tbinom{5^t(2+4d)+1}{1}\bi^{5^t(2+4d)}\Big(\big(5^t(2+4d)+1\big)5^{l+1}h\Big)-(\bi+5^{l+1}h)\\
&\equiv 5^{t+l+1}(4+8d)h~(\mathrm{mod}~5^{t+l+2}).
\end{align*}
Therefore,
\begin{align*}
B_{l+2}=\nu_5(b_{l+2})=\nu_5\big(\frac{f^2(\bi+5^{l+1}h)-(\bi+5^{l+1}h)}{5^{l+2}}\big)=t-1.
\end{align*}

If $0\leq l\leq t-3$, we obtain $A_{l+2}=t$ and $B_{l+2}=t-1$. By Definition (\ref{def: A_l and B_l}) of $B_{l+2}$, $B_{l+2}=t-1$ implies $B_{l+3}=t-2$, $B_{l+4}=t-3$, $\dots$, $B_{l+s+2}=t-s-1$ and $B_{l+s+3}=t-s-2$ for some $\frac{t-l-5}{2}<s\leq \frac{t-l-3}{2}$. Then, $B_{l+s+3}=t-s-2<t=A_{l+s+3}$ and $B_{l+s+3}=t-s-2<l+s+3$. By Proposition \ref{prop: lifts of splitting cycle for p>=3}, all the lifts of $\{\bi+5^{l+1}i,-\bi-5^{l+1}i\}$ at level $t+l+1$ grow forever. So, we obtain $4\cdot 5^{t-1}$ growing cycles of length 2 at level $t+l+1$.

For $l>t-3$, since $B_{l+2}=t-1<t=A_{l+2}$ and $B_{l+2}=t-1<l+2$, by Proposition \ref{prop: lifts of splitting cycle for p>=3}, all the lifts of $\{\bi+5^{l+1}i,-\bi-5^{l+1}i\}$ at level $t+l+1$ grow forever. So, we obtain $4\cdot 5^{t-1}$ growing cycles of length 2 at level $t+l+1$.

Therefore, $f(x)$ has $8\cdot 5^{t-1}$ growing cycles of length 1 and $4\cdot 5^{t-1}$ growing cycles of length 2 at level $\geq t+1$, which complete the proof.
\end{proof}

By Propositions \ref{prop:pZ_p attracting basin} and \ref{prop: p=5, m=1 (mod 5)}, we conclude that the following is true.

\begin{theorem}\label{thm:m=1(mod5)}
Let $f(x)=x^m$ over $\Z_5$ and assume $m\geq 2$ with $m\equiv 1~(\mathrm{mod}~5)$. Let $t=\nu_5(m-1)$.
\begin{enumerate}
\item If $m\equiv 1~(\mathrm{mod}~20)$, then the minimal decomposition of $\Z_5$ for $f(x)$ is
\begin{align*}
\Z_5=\{0,\pm 1,\pm \bi\}\bigsqcup \big(\bigcup_{l\geq 0}\bigcup_{a=0}^{5^{t-1}-1}\bigcup_{i=1}^4\bigcup_{j=1}^4 M_{l,a,i,j}\big)\bigsqcup (5\Z_5-\{0\}),
\end{align*}
where
\begin{align*}
M_{l,a,i,1}&=1+5^{l+1}i+5^{l+2}a+5^{t+l+1}\Z_5,\\
M_{l,a,i,2}&=-1+5^{l+1}i+5^{l+2}a+5^{t+l+1}\Z_5,\\
M_{l,a,i,3}&=\bi+5^{l+1}i+5^{l+2}a+5^{t+l+1}\Z_5~\mathrm{and}\\
M_{l,a,i,4}&=-\bi+5^{l+1}i+5^{l+2}a+5^{t+l+1}\Z_5.
\end{align*}
Here, $\{0,\pm 1,\pm \bi\}$ is the set of fixed points, $M_{l,a,i,j}$'s are the minimal components, and $5\Z_5-\{0\}$ is the attracting basin of the fixed point 0.

\item If $m\equiv 6~(\mathrm{mod}~10)$, then the minimal decomposition of $\Z_5$ for $f(x)$ is
\begin{align*}
\Z_5&=\{0,1\}\bigsqcup \big(\bigcup_{l\geq 0}\bigcup_{a=0}^{5^{t-1}-1}\bigcup_{i=1}^4 M_{l,a,i}\big)\\
&\quad \bigsqcup \big((5\Z_5-\{0\})\cup(2+5\Z_5)\cup(3+5\Z_5)\cup(4+5\Z_5)\big),
\end{align*}
where $M_{l,a,i}=1+5^{l+1}i+5^{l+2}a+5^{t+l+1}\Z_5$. Here, $\{0,1\}$ is the set of fixed points, $M_{l,a,i}$'s are the minimal components, and $(5\Z_5-\{0\})$ is the attracting basin of the fixed point 0 and $(2+5\Z_5)\cup(3+5\Z_5)\cup(4+5\Z_5)$ is the attracting basin of $1+5\Z_5$.

\item If $m\equiv 11~(\mathrm{mod}~20)$, then the minimal decomposition of $\Z_5$ for $f(x)$ is
\[ \Z_5=\{0,\pm 1,\pm\bi\}\bigsqcup \big(\bigcup_{l\geq 0}\bigcup_{a=0}^{5^{t-1}-1}\bigcup_{i=1}^4\bigcup_{j=1}^3 M_{l,a,i,j}\big) \bigsqcup (5\Z_5-\{0\}), \]
where
\begin{align*}
M_{l,a,i,1}&=1+5^{l+1}i+5^{l+2}a+5^{t+l+1}\Z_5,\\
M_{l,a,i,2}&=-1+5^{l+1}i+5^{l+2}a+5^{t+l+1}\Z_5~\mathrm{and}\\
M_{l,a,i,3}&=(\bi+5^{l+1}i+5^{l+2}a+5^{t+l+1}\Z_5)\\
&\quad \cup(-\bi-5^{l+1}i-5^{l+2}a+5^{t+l+1}\Z_5).
\end{align*}
Here, $\{0,\pm 1,\pm\bi\}$ is the set of periodic points, where $0,\pm 1$ are fixed points and $\pm\bi$ are two periodic points to each other. $M_{l,a,i,j}$'s are the minimal components, and $5\Z_5-\{0\}$ is the attracting basin.
\end{enumerate}
\end{theorem}

\begin{proof}
1. The lifts of $\{1+5^{l+1}i\}$, $\{-1+5^{l+1}i\}$, $\{\bi+5^{l+1}i\}$ and $\{-\bi+5^{l+1}i\}$ from level $l+2$ to level $t+l+1$ are of the forms $\{1+5^{l+1}i+5^{l+2}a\}$, $\{-1+5^{l+1}i+5^{l+2}a\}$, $\{\bi+5^{l+1}i+5^{l+2}a\}$ and $\{-\bi+5^{l+1}i+5^{l+2}a\}$, respectively, where $i\in\{1,2,3,4\}$ and $a\in\{0,1,\dots,5^{t-1}-1\}$. Note that when $t=1$, the value $a$ is only $0$. Hence from Propositions \ref{prop:pZ_p attracting basin} and \ref{prop: p=5, m=1 (mod 5)}, the statement is proved.

2. The lifts of $\{1+5^{l+1}i\}$ from level $l+2$ to  level $t+l+1$ are of the forms $\{1+5^{l+1}i+5^{l+2}a\}$, respectively, where $i\in\{1,2,3,4\}$ and $a\in\{0,\dots,5^{t-1}-1\}$.

Write $m=6+10d$, then $d\geq 0$. If $d$ is even, then $f(2)\equiv 4$, $f(3)\equiv 4$ and $f(4)\equiv 1$ (mod 5), which imply that $f(2+5\Z_5)\cup f(3+5\Z_5)\subset 4+5\Z_5$ and $f(4+5\Z_5)\subset 1+5\Z_5$. So, $(2+5\Z_5)\cup (3+5\Z_5)\cup (4+5\Z_5)$ lies in the attracting basin of $1+5\Z_5$.
If $d$ is odd, then $f(2)\equiv 1$, $f(3)\equiv 1$ and $f(4)\equiv 1$ (mod 5), which imply that $f(2+5\Z_5)\cup f(3+5\Z_5)\cup f(4+5\Z_5)\subset 1+5\Z_5$. So, $(2+5\Z_5)\cup (3+5\Z_5)\cup (4+5\Z_5)$ lies in the attracting basin of $1+5\Z_5$.
Hence from Propositions \ref{prop:pZ_p attracting basin} and \ref{prop: p=5, m=1 (mod 5)}, the statement is proved.

3. The lifts of $\{1+5^{l+1}i\}$, $\{-1+5^{l+1}i\}$, $\{\bi+5^{l+1}i,-\bi-5^{l+1}i\}$ from level $l+2$ to level $t+l+1$ are of the forms $\{1+5^{l+1}i+5^{l+2}a\}$, $\{-1+5^{l+1}i+5^{l+2}a\}$, $\{\bi+5^{l+1}i+5^{l+2}a,-\bi-5^{l+1}i-5^{l+2}a\}$, respectively, where $i\in\{1,2,3,4\}$ and $a\in\{0,\dots,5^{t-1}-1\}$. Hence from Propositions \ref{prop:pZ_p attracting basin} and \ref{prop: p=5, m=1 (mod 5)}, the statement is proved.
\end{proof}

Now we consider the case $m\equiv -1 ~(\mathrm{mod}~5)$.

\begin{proposition}\label{prop: p=5, m=4 (mod 5)}
Let $f(x)=x^m$ over $\Z_5$ and assume $m\geq 2$ with $m\equiv -1~(\mathrm{mod}~5)$. Let $t=\nu_5(m+1)$.
\begin{enumerate}
\item If $m\equiv 4~(\mathrm{mod}~10)$, then $f(x)$ has $2\cdot 5^{t-1}$ growing cycles of length 2 at every level $\geq t+1$.

\item If $m\equiv 9~(\mathrm{mod}~10)$, then $f(x)$ has $8\cdot 5^{t-1}$ growing cycles of length 2 at every level $\geq t+1$.
\end{enumerate}
\end{proposition}

We omit the proof of Proposition \ref{prop: p=5, m=4 (mod 5)} because it is similar to that of Propositions \ref{prop: p=3, m=-1 (mod 3)} and \ref{prop: p=5, m=1 (mod 5)}. By Propositions \ref{prop:pZ_p attracting basin} and \ref{prop: p=5, m=4 (mod 5)}, we conclude that the following is true.

\begin{theorem}\label{thm:m=-1(mod5)}
Let $f(x)=x^m$ over $\Z_5$ and assume $m\geq 2$ with $m\equiv -1~(\mathrm{mod}~5)$. Let $t=\nu_5(m+1)$.
\begin{enumerate}
\item If $m\equiv 4~(\mathrm{mod}~10)$, then the minimal decomposition of $\Z_5$ for $f(x)$ is
\begin{align*}
\Z_5&=\{0,1\}\bigsqcup \big(\bigcup_{l\geq 0}\bigcup_{a=0}^{5^{t-1}-1}\bigcup_{i=1}^2 M_{l,a,i}\big)\\
&\quad \bigsqcup \big((5\Z_5-\{0\})\cup(2+5\Z_5)\cup(3+5\Z_5)\cup(4+5\Z_5)\big),
\end{align*}
where $M_{l,a,i}=(1+5^{l+1}i+5^{l+2}a+5^{t+l+1}\Z_5) \cup\big((1+5^{l+1}i+5^{l+2}a)^{-1}+5^{t+l+1}\Z_5\big)$. Here, $\{0,1\}$ is the set of fixed points, $M_{l,a,i}$'s are the minimal components, and $5\Z_5-\{0\}$ and $(2+5\Z_5)\cup(3+5\Z_5)\cup(4+5\Z_5)$ are the attracting basin of 0 and $1+5\Z_5$, respectively.

\item If $m\equiv 9~(\mathrm{mod}~20)$, then the minimal decomposition of $\Z_5$ for $f(x)$ is
\begin{align*}
\Z_5&=\{0,\pm 1,\pm \bi\}\bigsqcup \Big(\big(\bigcup_{l\geq 0}\bigcup_{a=0}^{5^{t-1}-1}\bigcup_{i=1}^2\bigcup_{j=1}^4 M_{l,a,i,j}\big)\bigsqcup (5\Z_5-\{0\}),
\end{align*}
where
\begin{align*}
M_{l,a,i,1}&=(1+5^{l+1}i+5^{l+2}a+5^{t+l+1}\Z_5)\\
&\quad \cup\big((1+5^{l+1}i+5^{l+2}a)^{-1}+5^{t+l+1}\Z_5\big),\\
M_{l,a,i,2}&=(-1+5^{l+1}i+5^{l+2}a+5^{t+l+1}\Z_5)\\
&\quad \cup\big((-1+5^{l+1}i+5^{l+2}a)^{-1}+5^{t+l+1}\Z_5\big),\\
M_{l,a,i,3}&=(\bi+5^{l+1}i+5^{l+2}a+5^{t+l+1}\Z_5)\\
&\quad \cup\big((\bi+5^{l+1}i+5^{l+2}a)^{-1}+5^{t+l+1}\Z_5\big) ~\mathrm{and}\\
M_{l,a,i,4}&=(-\bi+5^{l+1}i+5^{l+2}a+5^{t+l+1}\Z_5)\\
&\quad \cup\big((-\bi+5^{l+1}i+5^{l+2}a)^{-1}+5^{t+l+1}\Z_5\big).
\end{align*}
Here, $\{0,\pm 1,\pm \bi\}$ is the set of fixed points, $M_{l,a,i,j}$'s are the minimal components, and $5\Z_5-\{0\}$ is the attracting basin of the fixed point 0.

\item If $m\equiv 19~(\mathrm{mod}~20)$, then the minimal decomposition of $\Z_5$ for $f(x)$ is
\begin{align*}
\Z_5&=\{0,\pm 1,\pm \bi\}\bigsqcup \Big(\big(\bigcup_{l\geq 0}\bigcup_{a=0}^{5^{t-1}-1}\bigcup_{i=1}^2\bigcup_{j=1}^2 M_{l,a,i,j}\big)\\
&\quad \cup \big(\bigcup_{l\geq 0}\bigcup_{a=0}^{5^{t-1}-1}\bigcup_{i=1}^4 M'_{l,a,i}\big)\Big)\bigsqcup (5\Z_5-\{0\}),
\end{align*}
where
\begin{align*}
M_{l,a,i,1}&=(1+5^{l+1}i+5^{l+2}a+5^{t+l+1}\Z_5)\\
&\quad \cup\big((1+5^{l+1}i+5^{l+2}a)^{-1}+5^{t+l+1}\Z_5\big),\\
M_{l,a,i,2}&=(-1+5^{l+1}i+5^{l+2}a+5^{t+l+1}\Z_5)\\
&\quad \cup\big((-1+5^{l+1}i+5^{l+2}a)^{-1}+5^{t+l+1}\Z_5\big)~\mathrm{and}\\
M'_{l,a,i}&=(\bi+5^{l+1}i+5^{l+2}a+5^{t+l+1}\Z_5)\\
&\quad \cup\big((\bi+5^{l+1}i+5^{l+2}a)^{-1}+5^{t+l+1}\Z_5\big).
\end{align*}
Here, $\{0,\pm 1,\pm \bi\}$ is the set of periodic points, where $0,\pm1$ are fixed points and $\pm\bi$ are two periodic points to each other. $M_{l,a,i,j}$'s and $M'_{l,a,i}$'s are the minimal components, and $5\Z_5-\{0\}$ is the attracting basin of the fixed point 0.
\end{enumerate}
\end{theorem}

\begin{proof}
1. The lifts of $\{1+5^{l+1}i,1-5^{l+1}i\}$ from level $l+2$ to level $t+l+1$ are of the form $\{1+5^{l+1}i+5^{l+2}a,(1+5^{l+1}i+5^{l+2}a)^{-1}\}$, where $i\in\{1,2\}$ and $a\in\{0,\dots,5^{t-1}-1\}$. Hence from Propositions \ref{prop:pZ_p attracting basin} and \ref{prop: p=5, m=4 (mod 5)}, the statement is proved.

2. The lifts of $\{1+5^{l+1}i,1-5^{l+1}i\}$, $\{-1+5^{l+1}i,-1-5^{l+1}i\}$, $\{\bi+5^{l+1}i, \bi-5^{l+1}i\}$ and $\{-\bi+5^{l+1}i,-\bi-5^{l+1}i\}$ from level $l+2$ to level $t+l+1$ are of the forms $\{1+5^{l+1}i+5^{l+2}a,(1+5^{l+1}i+5^{l+2}a)^{-1}\}$, $\{-1+5^{l+1}i+5^{l+2}a,(-1+5^{l+1}i+5^{l+2}a)^{-1}\}$, $\{\bi+5^{l+1}i+5^{l+2}a,-(\bi+5^{l+1}i+5^{l+2}a)^{-1}\}$ and $\{-\bi+5^{l+1}i+5^{l+2}a,(-\bi+5^{l+1}i+5^{l+2}a)^{-1}\}$, respectively, where $i\in\{1,2\}$ and $a\in\{0,\dots,5^{t-1}-1\}$. Hence from Propositions \ref{prop:pZ_p attracting basin} and \ref{prop: p=5, m=4 (mod 5)}, the statement is proved.

3. The lifts of $\{1+5^{l+1}i,1-5^{l+1}i\}$, $\{-1+5^{l+1}i,-1-5^{l+1}i\}$ and $\{\bi+5^{l+1}i,-\bi+5^{l+1}i\}$ from level $l+2$ to level $t+l+1$ are of the forms $\{1+5^{l+1}i+5^{l+2}a,(1+5^{l+1}i+5^{l+2}a)^{-1}\}$, $\{-1+5^{l+1}i+5^{l+2}a,(-1+5^{l+1}i+5^{l+2}a)^{-1}\}$ and $\{\bi+5^{l+1}i+5^{l+2}a,(\bi+5^{l+1}i+5^{l+2}a)^{-1}\}$, respectively, where $i\in\{1,2\}$ and $a\in\{0,\dots,5^{t-1}-1\}$. Hence from Propositions \ref{prop:pZ_p attracting basin} and \ref{prop: p=5, m=4 (mod 5)}, the statement is proved.
\end{proof}

Consider the cases $m\equiv 2~(\mathrm{mod}~5)$ and $m\equiv -2~(\mathrm{mod}~5)$. We notice from Theorem \ref{thm:m=1(mod5)} and \ref{thm:m=-1(mod5)} that the minimal decomposition heavily depends on the values $\nu_5(m-1)$ and $\nu_5(m+1)$. For the cases $m\equiv 2~(\mathrm{mod}~5)$ and $m\equiv -2~(\mathrm{mod}~5)$, our computation shows that there is a similar pattern for the minimal decomposition which heavily depends on the values $\nu_5(m-\bi)$ and $\nu_5(m+\bi)$. 

When we write $\bi=\sum_{i=0}^\infty a_i 5^i$, the value $a_i$ can be computed using the values $a_0, a_1, \dots, a_{i-1}$, by Hensel's lemma, in this way, the pattern of $a_i$ is irregular. Hence we do not have a general formula for the values $\nu_5(m -\bi)$ and $\nu_5(m +\bi)$ and so the minimal decomposition for this case can not be proved using previous methods. We need a different method which we do not know yet. 

We state our computational result as a conjecture for general $m$ and give an example for a specific $m$. 

\begin{conjecture}\label{conj: p=5, m=2,-2 (mod 5)}
Let $f(x)=x^m$ over $\Z_5$ and assume $m\geq 2$ with $m\equiv \pm2~(\mathrm{mod}~5)$. Let $t=\min\{\nu_5(m\mp\bi)\}$.
\begin{enumerate}
\item If $m\equiv \pm 2~(\mathrm{mod}~10)$, then $f(x)$ has $5^{t-1}$ growing cycles of length 4 at every level $\geq t+1$.

\item If $m\equiv \pm 7~(\mathrm{mod}~10)$, then $f(x)$ has $4\cdot 5^{t-1}$ growing cycles of length 4 at every level $\geq t+1$.
\end{enumerate}
\end{conjecture}

\begin{example}
  If $m\equiv 7~(\mathrm{mod}~10)$, then $\pm 1$ and $\pm\bi$ are fixed points. So, let $\{1\}$, $\{-1\}$, $\{\bi\}$ and $\{-\bi\}$ be cycles of length 1 at level $l+1$ for any $l\geq 0$. Then, the lifts to level $l+2$, are $\{b, b+5^{l+1},b+5^{l+1}\cdot 2,b+5^{l+1}\cdot 3,b+5^{l+1}\cdot 4\}$ with $b\in \{1,-1,\bi,-\bi \}$.
  
 If $t=1$, then those 4 lifts grow at level $t+l+1=l+2$. With Proposition \ref{prop:pZ_p attracting basin}, the minimal decomposition of $\Z_5$ for $f(x)$ is
\[ \Z_5=\{0,\pm 1,\pm\bi\}\bigsqcup \big(\bigcup_{l\geq 0}\bigcup_{j=1}^4 M_{l,j}\big) \bigsqcup (5\Z_5-\{0\}), \]
where
\begin{align*}
M_{l,1}&=\bigcup_{a=1}^4 (1+5^{l+1}a+5^{l+2}\Z_5),~M_{l,2}=\bigcup_{a=1}^4 (-1+5^{l+1}a+5^{l+2}\Z_5),\\
M_{l,3}&=\bigcup_{a=1}^4 (\bi+5^{l+1}a+5^{l+2}\Z_5)~\mathrm{and}~M_{l,4}=\bigcup_{a=1}^4 (-\bi+5^{l+1}a+5^{l+2}\Z_5).
\end{align*}
Here, $\{0,\pm 1,\pm\bi\}$ is the set of fixed points, $M_{l,j}$'s are the minimal components, and $5\Z_5-\{0\}$ is the attracting basin.

If $t=2$, we obtain $4\cdot 5^{t-1}=20$ growing lifts at level $t+l+1=l+3$, respectively. Those are 
$\{1+5^{l+1}+5^{l+2}i,1+5^{l+1}\cdot 7+5^{2l+2}+5^{l+2}\cdot 2i,1+5^{l+1}\cdot 24+5^{2l+2}+5^{l+2}\cdot 4i,1+5^{l+1}\cdot 18+5^{2l+2}\cdot 3+5^{l+2}\cdot 3i\}$, 
$\{-1+5^{l+1}+5^{l+2}i,-1+5^{l+1}\cdot 7+5^{2l+2}\cdot 4+5^{l+2}\cdot 2i,-1+5^{l+1}\cdot 24+5^{2l+2}\cdot 4+5^{l+2}\cdot 4i,-1+5^{l+1}\cdot 18+5^{2l+2}\cdot 2+5^{l+2}\cdot 3i\}$, 
$\{\bi+5^{l+1}+5^{l+2}i,\bi+5^{l+1}\cdot 7-5^{2l+2}\bi+5^{l+2}\cdot 2i,\bi+5^{l+1}\cdot 24-5^{2l+2}\bi+5^{l+2}\cdot 4i,\bi+5^{l+1}\cdot 18+5^{2l+2}\cdot 2\bi+5^{l+2}\cdot 3i\}$ and 
$\{-\bi+5^{l+1}+5^{l+2}i,-\bi+5^{l+1}\cdot 7+5^{2l+2}\bi+5^{l+2}\cdot 2i,-\bi+5^{l+1}\cdot 24+5^{2l+2}\bi+5^{l+2}\cdot 4i,-\bi+5^{l+1}\cdot 18-5^{2l+2}\cdot 2\bi+5^{l+2}\cdot 3i\}$ 
where $i\in\{0,\dots,4\}$. With Proposition \ref{prop:pZ_p attracting basin}, the minimal decomposition of $\Z_5$ for $f(x)$ is
\[ \Z_5=\{0,\pm 1,\pm\bi\}\bigsqcup \big(\bigcup_{l\geq 0}\bigcup_{i=0}^4\bigcup_{j=1}^4 M_{l,i,j}\big) \bigsqcup (5\Z_5-\{0\}), \]
where
\begin{align*}
M_{l,i,1}&=(1+5^{l+1}+5^{l+2}i+5^{l+3}\Z_5)\\
&\quad \cup(1+5^{l+1}\cdot 7+5^{2l+2}+5^{l+2}\cdot 2i+5^{l+3}\Z_5)\\
&\quad \cup(1+5^{l+1}\cdot 24+5^{2l+2}+5^{l+2}\cdot 4i+5^{l+3}\Z_5)\\
&\quad \cup(1+5^{l+1}\cdot 18+5^{2l+2}\cdot 3+5^{l+2}\cdot 3i+5^{l+3}\Z_5),\\
M_{l,i,2}&=(-1+5^{l+1}+5^{l+2}i+5^{l+3}\Z_5)\\
&\quad \cup(-1+5^{l+1}\cdot 7+5^{2l+2}\cdot 4+5^{l+2}\cdot 2i+5^{l+3}\Z_5)\\
&\quad \cup(-1+5^{l+1}\cdot 24+5^{2l+2}\cdot 4+5^{l+2}\cdot 4i+5^{l+3}\Z_5)\\
&\quad \cup(-1+5^{l+1}\cdot 18+5^{2l+2}\cdot 2+5^{l+2}\cdot 3i+5^{l+3}\Z_5)\big),\\
M_{l,i,3}&=(\bi+5^{l+1}+5^{l+2}i+5^{l+3}\Z_5)\\
&\quad \cup(\bi+5^{l+1}\cdot 7-5^{2l+2}\bi+5^{l+2}\cdot 2i+5^{l+3}\Z_5)\\
&\quad \cup(\bi+5^{l+1}\cdot 24-5^{2l+2}\bi+5^{l+2}\cdot 4i+5^{l+3}\Z_5)\\
&\quad \cup(\bi+5^{l+1}\cdot 18+5^{2l+2}\cdot 2\bi+5^{l+2}\cdot 3i+5^{l+3}\Z_5)~\mathrm{and}\\
M_{l,i,4}&=(-\bi+5^{l+1}+5^{l+2}i+5^{l+3}\Z_5)\\
&\quad \cup(-\bi+5^{l+1}\cdot 7+5^{2l+2}\bi+5^{l+2}\cdot 2i+5^{l+3}\Z_5)\\
&\quad \cup(-\bi+5^{l+1}\cdot 24+5^{2l+2}\bi+5^{l+2}\cdot 4i+5^{l+3}\Z_5)\\
&\quad \cup(-\bi+5^{l+1}\cdot 18-5^{2l+2}\cdot 2\bi+5^{l+2}\cdot 3i+5^{l+3}\Z_5).
\end{align*}
Here, $\{0,\pm 1,\pm\bi\}$ is the set of fixed points, $M_{l,i,j}$'s are the minimal components, and $5\Z_5-\{0\}$ is the attracting basin.
\end{example}

Like the examples, if we have a specific integer $m\geq 2$ with $m\equiv$ 2 or -3 (mod 5), then we obtain the value $t=\min\{\nu_5(m\mp\bi)\}$ and minimal components by Conjectures \ref{conj: p=5, m=2,-2 (mod 5)}.

Postal address of Jung and Kim: Department of Mathematics, Korea University, Anam-ro, Seongbuk-gu, Seoul, 02841, Republic of Korea

     Email address of Jung: \texttt{myunghyun.jung07@gmail.com}

     Email address of Kim: \texttt{kim.donggyun@gmail.com}

\end{document}